\documentclass{amsart}[12 pt]

\usepackage{latexsym}
\usepackage{amsmath,amssymb, amsfonts}
\usepackage{verbatim}
\usepackage{graphicx}
\usepackage{graphics}
\usepackage{geometry}
\usepackage{booktabs}
\usepackage{color}
\usepackage{bm}
\usepackage{longtable}
\usepackage{caption}
\usepackage{natbib,bm}
\usepackage[pdftex]{hyperref}
\usepackage{float}
\newtheorem{theorem}{Theorem}[section]
\newtheorem{corollary}{Corollary}[section]
\newtheorem{lemma}{Lemma}[section]

\newcommand{\be}{\begin{equation}}
\newcommand{\ee}{\end{equation}}

\newcommand{\B}{\beta}

\newcommand{\lt}{\left}
\newcommand{\rt}{\right}
\newcommand{\sxj}{\sigma^2_{x_j}}
\newcommand{\sij}{\sigma^2_{ij}}

\newcommand{\sdos}{\sigma^2}
\usepackage{enumerate}

\allowdisplaybreaks

\makeatletter
\def\@setauthors{%
	\begingroup
	\def\thanks{\protect\thanks@warning}%
	\trivlist
	\centering\footnotesize \@topsep30\p@\relax
	\advance\@topsep by -\baselineskip
	\item\relax
	\author@andify\authors
	\def\\{\protect\linebreak}%
	\authors%
	\ifx\@empty\contribs
	\else
	,\penalty-3 \space \@setcontribs
	\@closetoccontribs
	\fi
	\endtrivlist
	\endgroup
}
\makeatother
\begin{document}

\title[Data analysis for proficiency testing] {Data analysis for proficiency testing}

	\author[Reiko Aoki]{REIKO AOKI$  ^{(a)}  $}
	
	\author[Dorival Le\~ao]{DORIVAL LE\~AO$^{(b)}$}
	
	\author[Juan P.M. Bustamante]{JUAN P. M. BUSTAMANTE$^{(a)}$}
	
	\author[Filidor Vilca Labra]{FILIDOR VILCA LABRA$^{(c)}$}

\keywords{ultrastructural model, measurement error model, asymptotic theory, hypothesis testing, confidence region}
\maketitle

\begin{center}
$^{(a)} $ Instituto de Ci\^encias Matem\'aticas e de Computa\c c\~ao.

Universidade de S\~ao
Paulo, S\~ao Carlos - SP, Brazil

$^{(b)} $ Estatcamp Consultoria

S\~ao Carlos - SP, Brazil

$^{(c)} $ Instituto de Matem\'atica, Estat\'istica e Computa\c c\~ao Cient\'ifica

Universidade de Campinas, Campinas - SP, Brazil
\end{center}

\begin{abstract}Proficiency Testing (PT) determines the performance of individual laboratories for specific tests or measurements and it is used to monitor the reliability of laboratories measurements. PT plays a highly valuable role as it provides an objective evidence of the competence of the
participant
laboratories. In this paper, we propose a multivariate model to assess equivalence among laboratories
measurements in proficiency testing. Our method allow to include type B source
of variation and to deal with multivariate data, where the item under test is measured
at different levels. Although intuitive, the proposed model is nonergodic, which means that the asymptotic Fisher information matrix is random. As a consequence, a detailed asymptotic analysis was carried out to establish the strategy for comparing the results of the participating laboratories. To illustrate, we apply our method to analyze the data from the Brazilian Engine test group, PT program, where the power of an engine was measured by $8$ laboratories at several levels of rotation.
\end{abstract}
%

\section{Introduction}
Proficiency studies are conducted to evaluate the equivalence of laboratories measurements (see, \cite{ISOIEC170432010}). In these studies, a reference value of some measurand (the quantity to be measured) is determined and the results of all the laboratories are compared to this reference value. According to \cite{ISOIEC170432010} and \cite{EA418Guidance}, acrredited laboratories should assure the quality of test results by participating in proficiency testing programs. Various statistical techniques have been adopted to assess equivalence among laboratories measurements. These include classical statistical techniques such as paired t test, z-score, normalized error, repeated measures analysis of variance, Bland-Altman plot, see \cite{Altman1986}, \cite{linsinger1998influence}, \cite{rosario2008comparison}, \cite{ISOIEC170432010} and \cite{ISO135282015}, and references therein.

One critical point in most techniques is the fact that the type B source of variability (\cite{guide1995expression}) is not considered. In this direction, \cite{LEAOPINTO20091427} extended Jaech's model (\cite{jaech1985statistical}) to encompass the type B source of variation and they evaluated this model under elliptical distributions. \cite{toman2007bayesian} proposed a Bayesian Hierarquical model to encompass the type B source of variation into the model, see \cite{page2010using} for further developments of Bayesian Hierarquical model.

The approach to quantification of uncertainty of measurement is presented in \cite{guide1995expression}.  As discussed by \cite{gleser1998assessing}, the basic idea of \cite{guide1995expression} is to approximate a measurement equation $x = g(Z_1 , \cdots , Z_d)$, where $x$ denotes the measurand, $g$ is a known function and $Z_1, \cdots , Z_d$ denote the $d$ input quantities, by a first order Taylor series about the expected values of $Z_i$. The standard combined uncertainty is defined as the standard deviation of the probability distribution of $x$ based on this linear approximation. The expected value and the variance of each input quantity $Z_i$ may be based on measurements or any other information, such as the resolution of the measuring instrument. \cite{guide1995expression} defines two types of uncertainty evaluations. Type A is determined by the statistical analysis of a series of observations and type B is determined by other means, such as instruments specifications, correction factors or even data from additional experiments.

In spite of the amount of techniques to assess equivalence among laboratories measurements, these approaches consider the reference as a single value. In some proficiency studies the item under testing is measured at different levels. As one illustration, consider the measurement system to evaluate the engine power. In this case, the power (torque times rotation) is measured at different levels of rotation and as a result, we have one reference value for each level.

The first goal of this work is to propose a multivariate model to assess equivalence among laboratories measurements considering that the item under testing is measured at different levels. Let $y_{ijk}$ represent the $kth$ replica of the measurement of
the item under testing (engine power) at the $jth$ level (engine rotation) measured by the
$ith$ laboratory. We denote by $x_j$ the true unobserved value of the item under testing (engine
power) at the $jth$ point (engine rotation) for every $k=1,\cdots,n_i$,
$j=1,\cdots,m$, $i=1,\cdots,p$.

As considered by \cite{guide1995expression}, no measurement can perfectly determine the quantity to be measured $x_j$, known as measurand. In order to capture the inaccuracies and imprecisions arising from the measurements, we assume that $y_{ijk}$ satisfies
the linear ultrastructural relationship with the true (unobserved) value
$x_j$. We  denote by  $Y_{ijk}$ the observed value (subject to
measurement error) of the $kth$ replica of the measurement of the item under testing
(engine power) at the  $jth$ level (engine rotation) obtained by the
$ith$ laboratory. In this case, the proposed model is represented as follows
\begin{equation}\label{mod}
y_{ijk}=\alpha_{i} + \beta_{i} x_{j},
\end{equation}
\[
Y_{ijk} = y_{ijk}
 + e_{ijk},
 \]
where  $\mathbb{E}(e_{ijk})=0,$ $Var(e_{ijk})=\sigma^2_{ij}$, $\mathbb{E}(x_j)=\mu_{x_j},$
$Var(x_j)=\sigma^2_{x_j},$ $x_j$ independent of $e_{ijk}$,  for every $k=1,\cdots,n_i$, $i=1,\cdots,p$ and  $j=1,\cdots,m$.

Comparative calibration models are typically used in comparing different ways of measuring the same unknown quantity in a group of several available items, see \cite{Barnett1969}, \cite{Theobald1978}, \cite{Kimura1992}, \cite{cheng1997statistical} and \cite{Gimenez2014}.
The main difference between the proposed ultrastructural model (\ref{mod}) and the comparative calibration model is related with the available items. In a proficiency testing, the same item (engine) is measured by all participant laboratories instead of several independent items. Hence, there is a natural dependency among all measurements of the same level (rotation) of the item (engine). As a conclusion, the usual comparative calibration model is not appropriate to evaluate the performance of participant laboratories in a PT program.

In the proposed model (\ref{mod}), $x_j$ represents the measurand with mean $\mu_{x_j}$ and variance $\sigma^2_{x_j}$ at the $jth$ level. In general, the parameter $\sigma^2_{x_j}$ is determined by one expert laboratory  during the stability study which is conducted to guarantee the stability of the item under testing.  In our example, the GM power train developed one standard engine and, during the stability study, evaluated its natural variability $(\sigma^2_{x_j})$.


One of the basic element in all PT is the evaluation of the performance of each participant. In order to do so, the PT provider has to determine a reference value, which can be obtained in two different manners. One is to employ a reference laboratory and the other is to use a consensus value. We will use the reference laboratory strategy. Without loss of generality we consider that the first laboratory is the reference laboratory. In this case, we set $\alpha_1=0$ and $\beta_1=1$, then we have
\be \label{mod1}
Y_{1jk}=x_j+e_{1jk}, \quad j = 1,\cdots,m ~ ~ \text{and}~ ~  k=1,\cdots,n_1.
\ee
Here $e_{1jk}$ is the measurement error corresponding to the reference laboratory. The variance $\sigma^2_{1j}$ is determined by the combined variance calculated and provided by the reference laboratory following the protocol proposed by \cite{guide1995expression}. In the sequel, the participant laboratories measurements are given by
\begin{equation} \label{mod2}
Y_{ijk}=\alpha_{i} + \beta_{i} x_{j}+e_{ijk}, \quad  k=1,\cdots,n_i,  ~  j = 1,\cdots,m ,
\end{equation}
where $\alpha_i$ describes the additive bias and $\beta_i$ describes the multiplicative bias of laboratory $i=2, \cdots , p$. In the same way $e_{ijk}$ represents the measurement error associated with $ith$ laboratory. The variance $\sigma^2_{ij}$ is determined by the combined variance calculated and provided by the $ith$ laboratory following the protocol proposed by \cite{guide1995expression}

Considering the proposed ultrastructural measurement error model, the second goal of this work is to develop a test to evaluate the competence of the group of laboratories and also the competence of individual laboratories with respect to the reference laboratory. In this case, we provide statistical testing hypothesis to evaluate additive and multiplicative bias. Finally, we propose one graphical analysis to assess the equivalence of the measurements of the $ith$ laboratory with respect to the measurements of the reference laboratory.

Besides the fact that the proposed model is simple and intuitive, it presents some interesting properties. As we have only one item under testing for the entire PT program, the true unobserved value $x_j$  of the item under testing (engine power) at the $jth$ level (rotation) is also the same during the PT program. As a consequence, there is a natural dependency among all measurements at the same $jth$ level. This fact yields that the observed Information matrix converges in probability to a random matrix with components related to the mean of the measurand $(\mu_{x_j})$ being null. Thus, it is not possible to obtain consistent estimate for the parameter $\mu_{x_j}$. Subsequently, the usual asymptotic theory is not applicable to the proposed ultrastructural measurement error model.

By assuming that the measurement errors have normal distribution, we will apply the smoothness of the likelihood function to derive one suitable asymptotic theory to the ultrastructural measurement error model, as developed by \cite{weiss1971asymptotic}, \cite{weiss1973asymptotic} and \cite{sweeting1980uniform}. As a consequence of the asymptotic theory developed in Section 3, we will propose a Wald type test to evaluate the bias parameters. Moreover, we will apply the Wald statistics to develop a graphical analysis to assess the competence of each participant laboratory with respect to the reference laboratory.

In Section 2 we describe the model
and obtain the Score function, as well as, the observed information
matrix in closed form expressions. Moreover, we develop the EM
algorithm to obtain the maximum likelihood estimates (MLE) of the parameters. In
Section 3 we develop the asymptotic theory to assess the equivalence among laboratories measurements in proficiency testing.  Tests for the composite hypothesis
and confidence regions are obtained in Section 4.
Next, in Section 5, we perform a simulation study considering different number of replicas, nominal values and parameter values. In Section 6 we apply the developed methodology to the real data set collected
to perform a proficiency study. Finally,  we discuss the obtained results in Section 7.

\section{The model}
Considering the model defined by (\ref{mod}), (\ref{mod1} and \ref{mod2}), and the engine power illustration, the
covariance between the observations taken at the same value of the
engine rotation by the reference laboratory ($ith$ laboratory) is
given by $\sxj$ ($\beta_i^2\sxj$) and the covariance between the
observations of the reference laboratory and the $ith$ laboratory at
the same value of the engine rotation is given by $\beta_i \sxj$,
while the covariance between the observations of the $ith$ and $hth$
laboratory at the $jth$ engine rotation is given by $\beta_i \beta_h
\sxj$, that is: \[ cov(Y_{1jk},Y_{1jl})=\sxj;
cov(Y_{ijk},Y_{ijl})=\beta_i^2\sxj; cov(Y_{1jk},Y_{ijl})=\beta_i
\sxj;cov(Y_{ijk},Y_{hjl})=\beta_i \beta_h \sxj;\] $i,h=2,\cdots,p;$
$j=1,\cdots,m$; $k,l=1,\cdots,n_i$.

Let ${\bf Y}_{1j}=(Y_{1j1},\cdots,Y_{1jn_1})^T$ and  ${\bf Y}_{ij}=(Y_{ij1},\cdots,Y_{ijn_i})^T$ represent, respectively, the
measurements of the engine power of the reference laboratory  and the $ith$ laboratory
at the $jth$ engine rotation value, ${\bf Y}_j^n=({\bf Y}_{1j}^{\top},\cdots,{\bf
Y}_{pj}^{\top})^{\top},$ the measurements of all the laboratories at the $jth$ engine rotation value and finally,
${  \bf Y}^n=({\bf Y}_{1}^{n \top},\cdots,{\bf Y}_{m}^{n \top})^{\top}$ the observed data
with
$n=\sum_{i=1}^p n_i$. Then,  assuming that ${\bf Y}_j^n\sim
N_n(\mbox{\boldmath$\mu$}_j,\mbox{\boldmath$\Sigma$}_j)$, where
$\mbox{\boldmath$\mu$}_j=(\mbox{\boldmath$\mu$}_{1j}^T,\cdots,\mbox{\boldmath$\mu$}_{pj}^T)^T=
\mbox{\boldmath$\alpha$}+\mu_{x_j} \mbox{\boldmath$\beta$}$ and
$\mbox{\boldmath$\Sigma$}_j=D(\mbox{\boldmath$\sigma$}^2_j)+\sigma^2_{x_j}
\mbox{\boldmath$\beta$}\mbox{\boldmath$\beta$}^{\top},$ with
$\mbox{\boldmath$\alpha$}=$ $({\bf 0}_{n_1}^{\top},$ $\alpha_2{\bf
1}_{n_2}^{\top},\cdots,\alpha_p{\bf 1}_{n_p}^{\top})^{\top},$
$\mbox{\boldmath$\beta$}=({\bf 1}_{n_1}^{\top},\beta_2{\bf
1}_{n_2}^{\top},\cdots,\beta_p{\bf 1}_{n_p}^{\top})^{\top},$
$\mbox{\boldmath$\sigma$}^2_j=(\sigma^2_{1j}{\bf
1}_{n_1}^{\top},\cdots,\sigma^2_{pj}{\bf 1}_{n_p}^{\top})^{\top},$
${\bf 0}_{n_1}$ (${\bf
1}_{n_i}$) denoting a vector composed by $n_1$ zeros ($n_i$ one's) and
$D({\bf a})$ denoting the diagonal matrix with the diagonal elements
given by ${\bf a}$, we have,
${\bf Y}_{1j}\sim
N_{n_1}(\mbox{\boldmath$\mu$}_{1j},\mbox{\boldmath$\Sigma$}_{11j})$ and ${\bf Y}_{ij}\sim
N_{n_i}(\mbox{\boldmath$\mu$}_{ij},\mbox{\boldmath$\Sigma$}_{iij}),$ $i=2,\cdots,p$,
$j=1,\cdots,m$, where $\mbox{\boldmath$\mu$}_{1j}= \mu_{x_j} {\bf 1}_{n_1}$,
$\mbox{\boldmath$\Sigma$}_{11j}=\sigma^2_{1j} {\bf I}_{n_1}+\sxj {\bf
1}_{n_1}{\bf 1}_{n_1}^T$,
 $\mbox{\boldmath$\mu$}_{ij}= (\alpha_i+\beta_i \mu_{x_j}) {\bf 1}_{n_i}$,
$\mbox{\boldmath$\Sigma$}_{iij}=\sigma^2_{ij} {\bf I}_{n_i}+\beta_i^2
\sxj {\bf 1}_{n_i}{\bf 1}_{n_i}^T$, with ${\bf I}_{n_i}$   denoting the identity matrix of size $n_i$.
Furthermore,
\[
f_{{\bf Y}_j^n}({\bf y}_j^n,\mbox{\boldmath$\theta$})=(2\pi)^{-\frac{n}{2}}\mid
\mbox{\boldmath$\Sigma$}_j\mid^{-\frac{1}{2}} exp\left\{
-\frac{1}{2} ({\bf
y}_j^n-\mbox{\boldmath$\mu$}_j)^{\top}\mbox{\boldmath$\Sigma$}_j^{-1}({\bf
y}_j^n-\mbox{\boldmath$\mu$}_j)\right\},~j=1,\cdots,m\]  and
\[
f_{{\bf Y}^n}({\bf y}^n, \mbox{\boldmath$\theta$} )=\prod_{j=1}^m f_{{\bf
Y}_j^n}({\bf y}_j^n, \mbox{\boldmath$\theta$})~{\rm with}~\mbox{\boldmath$\theta$}=(\mu_{x_1},\cdots,\mu_{x_m},\alpha_2,\cdots,\alpha_p,\beta_2,\cdots,\beta_p)^T=(\theta_1,\cdots,
\theta_{m+2(p-1)})^T.
\]
The log-likelihood function is given by
\be \label{loglike}
L^n(\mbox{\boldmath$\theta$})=log f_{{\bf Y}^n}({\bf y}^n , \mbox{\boldmath$\theta$})=-\frac{m
n}{2}log(2\pi)-\frac{1}{2}\sum_{j=1}^m log
a_j^n-\frac{1}{2}\sum_{j=1}^m \sum_{i=1}^p n_i log
(\sigma^2_{ij})-\frac{1}{2}\sum_{j=1}^m Q_j^n,\ee where
$a_j^n=1+\sigma^2_{x_j}\mbox{\boldmath$\beta$}^{\top}D^{-1}(\mbox{\boldmath$\sigma$}^2_j)\mbox{\boldmath$\beta$}$
and $Q_j^n=({\bf
y}_j^n-\mbox{\boldmath$\mu$}_j)^{\top}\mbox{\boldmath$\Sigma$}_j^{-1}({\bf
y}_j^n-\mbox{\boldmath$\mu$}_j)$, $j=1,\cdots,m$.

After algebraic manipulations, the elements of the score function, $U^n(\mbox{\boldmath$\theta$})=\frac{\partial
L^n(\mbox{\boldmath$\theta$})}{\partial\mbox{\boldmath$\theta$}}$, denoted by $U^n_{\theta_q}$, $q=1,\dots, m+2(p-1)$ is given by:
\[ U^n_{\mu_{x_j}}=
\frac{M_j^n}{a_j^n}-\frac{\mu_{x_j}}{\sxj},
~ j=1,\cdots,m;
U^n_{\alpha_i}=-\sum_{j=1}^{m}\frac{1}{\sigma^2_{ij}}\lt[\frac{n_i \beta_i M_j^n\sxj}{a_j^n}-D_{ij}^n\rt],\]
\[
U^n_{\beta_i}=-\sum_{j=1}^{m}\frac{\sxj}{a_j^n \sigma^2_{ij}}\lt\{ n_i
\beta_i+M_j^n\lt[\frac{n_i \beta_i \sxj M_j^n}{a_j^n}-D_{ij}^n\rt]\rt\},~i=2,\cdots,p,\] with
$M_j^n=\frac{\mu_{x_j}}{\sxj}+\frac{{\bf Y}_{1j}^T {\bf
1}_{n_1}}{\sigma^2_{1j}}+\sum_{i=2}^p\frac{\beta_i D_{ij}^n}{\sij}$ and $D_{ij}^n=({\bf
Y}_{ij}-\alpha_i {\bf 1}_{n_i})^T {\bf 1}_{n_i}.$

Subsequently, the observed  information matrix, $\frac{J^n(\mbox{\boldmath$\theta$})}{n}=-\frac{1}{n}\frac{\partial^2
L^n(\mbox{\boldmath$\theta$})}{\partial\mbox{\boldmath$\theta$}\partial\mbox{\boldmath$\theta$}^T}$ was obtained in closed form expressions and  can be found in Appendix B.1.
\subsection{EM-algorithm}
In this subsection we are going to outline the EM-algorithm (\cite{Dempster1977}) used to
obtain the estimates of the parameters. In measurement error models, if the latent data $x_j$, $j=1,\cdots,m$, is introduced to augment
the observed data, the maximum likelihood estimates (MLE) of the parameters based on the augmented data (complete data) become easy to obtain.
Considering the model defined by (\ref{mod}) and (\ref{mod1})  and the
observed data for the $jth$ engine rotation value, ${\bf Y}_j^n=({\bf
Y}_{1j}^{\top},\cdots,{\bf Y}_{pj}^{\top})^{\top}$, we augment ${\bf
Y}_j^n$ by considering the unobserved data $x_j$. Then, the complete
data for the $jth$ engine rotation value is given by ${\bf
Y}_{jc}^n=(x_j, {\bf Y}_j^{nT})^T$,
with
$E\left(\mbox{\boldmath$Y$}
_{jc}^n\right)=\mbox{\boldmath$\mu$}_{jc}=(\mu_{x_j},\mbox{\boldmath$\mu$}_j^T)^T$ and the
covariance matrix given by
$\mbox{\boldmath$\Sigma$}_{jc}=\lt(\begin{array}{cc}
\sxj&\mbox{\boldmath$\Sigma$}_{12j}\\
\mbox{\boldmath$\Sigma$}_{21j}&\mbox{\boldmath$\Sigma$}_j\\
\end{array}\rt)$, $j=1,\cdots,m$, with
$\mbox{\boldmath$\Sigma$}_{12j}=\mbox{\boldmath$\Sigma$}_{21j}^T=\sxj\mbox{\boldmath$\beta$}^T,$
$\mbox{\boldmath$\mu$}_j$ and
$\mbox{\boldmath$\Sigma$}_j$ as given in Section 2.
Furthermore, ${\bf Y}_{jc}^n\sim
N_{n+1}(\mbox{\boldmath$\mu$}_{jc}, \mbox{\boldmath$\Sigma$}_{jc})$ and let ${\bf Y}_{c}^n=({\bf Y}_{1c}^{n T},\cdots,{\bf Y}_{mc}^{n T})^T$, then
\[
f_{\mbox{\boldmath$Y$}_c^n}(\mbox{\boldmath$y$}_c^n)=\prod_{j=1}^m f_{\mbox{\boldmath$Y$}_{jc}^n}(\mbox{\boldmath$y$}_{jc}^n)
\]
\[=(2\pi)^{-\frac{m(n+1)}{2}}\lt[\prod_{j=1}^m\lt(\sxj
(\sigma^2_{1j})^{n_1}\cdots(\sigma^2_{pj})^{n_p}\rt)^{-\frac{1}{2}}
\rt]exp\lt\{-\sum_{j=1}^m\frac{1}{2}({\bf
Y}_{jc}^n-\mbox{\boldmath$\mu$}_{jc})^T
\mbox{\boldmath$\Sigma$}_{jc}^{-1}({\bf
Y}_{jc}^n-\mbox{\boldmath$\mu$}_{jc})\rt\}.
\]
It follows that the log likelihood function of the complete data is
given by
\[
L_c(\mbox{\boldmath$\theta$})=cte-\frac{1}{2}\sum_{j=1}^m(log\sxj+\sum_{i=1}^p
n_i
log\sij)-\frac{1}{2}\lt[\sum_{j=1}^m\frac{(x_j-\mu_{x_j})^2}{\sxj}+\rt.\]
\[\lt.\sum_{j=1}^m
\sum_{k=1}^{n_1}\frac{(Y_{1jk}-x_j)^2}{\sigma^2_{1j}}+\sum_{j=1}^m\sum_{i=2}^p\sum_{k=1}^{n_i}\frac{(Y_{ijk}-\alpha_i-\beta_i
x_j)^2}{\sij}\rt],\]
which is much simpler than (\ref{loglike}).
Given the estimates of $\mbox{\boldmath$\theta$}$ in the $(r-1)th$
iteration, $\mbox{\boldmath$\theta$}^{(r-1)}$, the E step consists in
the obtention of the expectation of the complete data log-likelihood
function,
 $L_c(\mbox{\boldmath$\theta$})$, with respect to the conditional distribution
 of ${\bf x}=(x_1,\cdots,x_m)^T$ given the observed data ,$\textbf{Y}^n$, and
$\mbox{\boldmath$\theta$}^{(r-1)}$. The M step consists in the
maximization  of the function obtained in the E step with respect to
$\mbox{\boldmath$\theta$}$, which gives the estimates of the
parameters for the next iteration,
$\mbox{\boldmath$\theta$}^{(r)}.$ Each iteration of the EM
algorithm increments the log-likelihood function of the observed
data $L^n(\mbox{\boldmath$\theta$})
$, i.e.,
$L^n(\mbox{\boldmath$\theta$}^{(r-1)})
\leq
L^n(\mbox{\boldmath$\theta$}^{(r)})
.$ When the
likelihood function of the complete data belongs to the exponential
family, the implementation of the EM algorithm is usually simple. In
our case, the E step consists in the obtention of
$\mathbb{E}_{\mbox{\boldmath$\theta$}^{(r-1)}}\lt[x_j/{\bf Y}^n_j\rt]$
and $\mathbb{E}_{\mbox{\boldmath$\theta$}^{(r-1)}}\lt[x_j^2/{\bf Y}^n_j\rt]$, $j=1,\cdots,m$. In the M step, we maximize
the log-likelihood function of the complete data where the values of
the sufficient statistics were substituted by the expected values
obtained in the E step.

\medskip
The EM-algorithm for the model defined by (\ref{mod}) and (\ref{mod1}) may be summarized as follows.
\medskip

E-step: considering the properties of the multivariate normal
 distribution, the E step consists in the obtention of
\begin{eqnarray*}
{\widehat{x}_j}^{(r)}=
\mathbb{E}_{\mbox{\boldmath$\theta$}^{(r-1)}}\lt[x_j/{\bf Y}_j^n\rt]=
\frac{\sxj M_j^{n(r-1)}}{a_j^{n(r-1)}},
\end{eqnarray*}
where $M_j^{n(r-1)}=\lt[\frac{\mu_{x_j}^{(r-1)}}{\sxj}+\sum_{k=1}^{n_1}
\frac{Y_{1jk}}{\sigma^2_{1j}}+\sum_{i=2}^p\frac{\beta_i^{(r-1)}}{\sij}(\sum_{k=1}^{n_i}
Y_{ijk}-n_i \alpha_i^{(r-1)})\rt]$
and
\begin{eqnarray*}
\widehat{x^2_j}^{(r)}&=&\mathbb{E}_{\mbox{\boldmath$\theta$}^{(r-1)}}\lt[x_j^2/{\bf Y}^n_j\rt]=\frac{\sxj}{a_j^{n(r-1)}}+\left(\hat
x^{(r)}_j\right)^2,
\end{eqnarray*}
with $a_j^{n(r-1)}$ representing the value of $a_j^n$ evaluated at
$\mbox{\boldmath$\theta$}^{(r-1)}$.
\medskip

M-step: the M-step consists in the obtention of

\[\widehat{\mu}_{x_j}^{(r)}={\widehat{x}_j}^{(r)},~j=1,\cdots,m,\]
\[
\widehat{\beta_i}^{(r)}=\frac{\lt(\sum_{j=1}^m
\frac{{\widehat{x}_j}^{(r)}}{\sij} \sum_{k=1}^{n_i}Y_{ijk}
\rt)\lt(\sum_{j=1}^m
\frac{1}{\sij}\rt)-\lt(\sum_{j=1}^m\frac{{\widehat{x}_j}^{(r)}}{\sij}
\rt)\lt(\sum_{j=1}^m\frac{1}{\sij} \sum_{k=1}^{n_i}Y_{ijk}
\rt)}{n_i\lt[\lt(\sum_{j=1}^m\frac{\widehat{x^2_j}^{(r)}}{\sij}
\rt)\lt( \sum_{j=1}^m \frac{1}{\sij}\rt)-\lt( \sum_{j=1}^m
\frac{{\widehat{x}_j}^{(r)}}{\sij}\rt)^2 \rt]}~{\rm and}
\]
\[
\widehat{\alpha_i}^{(r)}=\frac{\lt(\sum_{j=1}^m\frac{1}{\sij}
\sum_{k=1}^{n_i}Y_{ijk}-n_i\widehat{\beta}_i^{(r)}
\sum_{j=1}^m\frac{ {\widehat{x}_j}^{(r)}}{\sij}\rt)}{n_i
\sum_{j=1}^m\frac{1}{\sij}},~i=2,\cdots,p.
\]

Notice that  closed form expressions were obtained for all the
expressions in the M-step, which means that this procedure will be
computationally inexpensive and also it is very simple to implement.
\section{Asymptotic Theory}
In this section, we will develop the asymptotic theory necessary to prove the consistency and the asymptotic distribution of the MLE regarding the bias parameters.  In the sequel, we will apply the regularity properties of the likelihood function to establish the asymptotic results, as proposed by \cite{weiss1971asymptotic}, \cite{weiss1973asymptotic} and \cite{sweeting1980uniform}. For a given $\mbox{\boldmath$\theta$}$, we define $$\mathbb{P}^n_{\mbox{\boldmath$\theta$}, (x_1, \cdots , x_m)} (B^n) = \int_{B^n} f_{{\bf Y}^n_c}({\bf y}^n_c, \mbox{\boldmath$\theta$} ) d  {\bf y}_c^n , $$ for every $B^n \in \beta (\mathbb{R}^{m(n+1)})$, where $\beta (\mathbb{R}^{m(n+1)})$ is the Borel $\sigma$-algebra. By applying Kolmogorov extension theorem, there exists a unique probability $\mathbb{P}_{\mbox{\boldmath$\theta$}, (x_1, \cdots , x_m)}$ defined on $(\mathbb{R}^\infty, \beta( \mathbb{R}^\infty))$ such that $$\mathbb{P}_{\mbox{\boldmath$\theta$}, (x_1, \cdots , x_m)} (B^n \times \mathbb{R} \times \mathbb{R} \times \cdots ) = \mathbb{P}^n_{\mbox{\boldmath$\theta$}, (x_1, \cdots , x_m)} (B^n),$$ for every $B^n \in \beta (\mathbb{R}^{m(n+1)})$. The marginal distribution of the observed data will be denoted by $\mathbb{P}_{\mbox{\boldmath$\theta$}}$ and the marginal distribution of the unobserved variables $(x_1, \cdots ,x_m)$ will be denoted by $\mathbb{P}_{(x_1, \cdots , x_m)}$. We say that $n \rightarrow \infty$ when $n_i \rightarrow \infty$ and $\frac{n_i}{n} \rightarrow w_i$ where $w_i$ is a positive constant for every $i=1, \cdots , p$.

Let $M^{\ell}$ be the space of all $\ell \times \ell$ matrices. The norm $\parallel A \parallel$ of the matrix $A$ is $\parallel A \parallel = \max \{ \mid A_{hs} \mid: h,s=1 \cdots , \ell\}$. A sequence of matrices $\{A^u: u=1, 2, \cdots \}$ converges to a limit $A$ if, and only if, $\parallel A^u - A \parallel \rightarrow 0$. If the matrix $A$ is positive definite, we write $A > 0$. In this case, $A^{1/2}$ denotes the symmetric positive square root of $A$. In the same way, for a given vector ${\bf v}=(v_1, \cdots , v_\ell) \in \mathbb{R}^\ell$, we consider the norm of ${\bf v}$ as follows $\mid {\bf v} \mid = \max \{ \mid v_1 \mid, \cdots , \mid v_\ell \mid\}$.

We denote by $B_c(n , \mbox{\boldmath$\theta$})$ the set of all vectors $\mbox{\boldmath$\psi$} \in \mathbb{R}^{m+2(p-1)}$ such that $\sqrt{n} \mid \mbox{\boldmath$\psi$} - \mbox{\boldmath$\theta$} \mid \leq c$, where $c$ is a positive constant. Moreover, we denote by $R_c(n , \mbox{\boldmath$\theta$})$ the set of all random vectors $\mbox{\boldmath$\phi$}$ with values in $\mathbb{R}^{m+2(p-1)}$ such that $\sqrt{n} \mid \mbox{\boldmath$\phi$} - \mbox{\boldmath$\theta$} \mid \leq c$. Here, we take the random vector $\mbox{\boldmath$\phi$}$ as function of the observed data ${\bf Y}^n$.

\begin{lemma} \label{conv_obs} For any sequences $\{\mbox{\boldmath$\psi$}^n: \mbox{\boldmath$\psi$}^n \in B_c(n,\mbox{\boldmath$\theta$}), ~ n \geq 1\}$ and $\{\mbox{\boldmath$\phi$}^n: \mbox{\boldmath$\phi$}^n \in R_c(n,\mbox{\boldmath$\theta$}), ~ n \geq 1\}$ there exist a  positive semidefinite random matrix $W(\mbox{\boldmath$\theta$})$ such that
\[
\mathbb{P}_{\mbox{\boldmath$\psi$}^n, (x_1, \cdots , x_m)} \left[  \parallel \frac{1}{n} J^n({\bm \phi}^n) - W(\mbox{\boldmath$\theta$}) \parallel \geq \epsilon \right] \rightarrow 0, \quad n \rightarrow \infty.
\]

\end{lemma}
\begin{proof} See Appendix B.2.

\end{proof}

The random matrix $W(\mbox{\boldmath$\theta$})$ has two important features. First, every component associated with $\mu_{x_j}$ is null, it means that we do not have enough information to estimate $\mu_{x_j}$ in a consistent way. This is a consequence of the fact that, for any level $j=1, \cdots, m$,
the same item (engine) is measured by all the laboratories
 under the same conditions. Second, the matrix $W(\mbox{\boldmath$\theta$})$ is random and the model is considered nonergodic. As a consequence the score random process $U^n(\mbox{\boldmath$\theta$})$ is nonergodic. Furthermore, the components associated with $\mu_{x_j}$ satisfy 

\begin{equation} \label{null_inform_VV}
\mathbb{P}_{\mbox{\boldmath$\psi$}^n, (x_1, \cdots , x_m)} \left[ \left|\frac{1}{\sqrt{n}} U^n_{\mu_{x_j}} (\mbox{\boldmath$\theta$}) \right| \geq \epsilon \right] \rightarrow 0, \quad \epsilon > 0, ~ ~ j=1,2,\cdots , m.
\end{equation} We will denote by $\tilde{U}^n(\mbox{\boldmath$\theta$})$ and $\tilde{J}^n(\mbox{\boldmath$\theta$})/n$ the score vector and the observed information matrix without the components involving $\mu_{x_j}$, respectively. We also denote by $\tilde{W}(\mbox{\boldmath$\theta$})$ the random matrix $W(\mbox{\boldmath$\theta$})$ without the components involving $\mu_{x_j}$. Moreover, we denote the bias components of the vector of parameters by $\tilde{\mbox{\boldmath$\theta$}}=(\alpha_2,\cdots,\alpha_p,\beta_2,\cdots,\beta_p)^T \in \mathbb{R}^{2(p-1)}$ and $\hat{\tilde{\mbox{\boldmath$\theta$}}}^n$ the related MLE. Furthermore, it follows from the Appendix B.2 that the random matrices $W(\mbox{\boldmath$\theta$})$ and $\tilde{W} (\mbox{\boldmath$\theta$})$ depend only on the bias components of the vector of parameters. As a consequence, from now on,  we will denote $W(\mbox{\boldmath$\theta$})$ and $\tilde{W} (\mbox{\boldmath$\theta$})$ by $W (\tilde{\mbox{\boldmath$\theta$}})$
and $\tilde{W} (\tilde{\mbox{\boldmath$\theta$}})$, respectively.

Let $\{g^n : n \geq 1\}$ be a sequence of real continuous function defined on a metric space, we say that $g^n (\tau)$ converges uniformly in $\tau$ to $g(\tau)$ if $g^n (\tau^n) \rightarrow g(\tau)$ for every sequence $\tau^n \rightarrow \tau$. Let $\lambda_\tau$ and $\{\lambda_{\tau}^n : n \geq 1 \}$ be probabilities defined on the Borel subsets of a metric space depending on the arbitrary parameter $\tau$, and let $C$ be the space of real bounded uniformly continuous functions.  We shall say that $\lambda^n_{\tau} \Rightarrow_u \lambda_\tau$ uniformly if
$$ \int g d \lambda_{\tau}^n \rightarrow \int g d \lambda_{\tau} \quad \text{uniformly in} ~ \tau, \quad \text{for all} ~ g \in C.$$ If $Q$ is a metric space and $\tau \in Q$, the family $\lambda_\tau$ of probabilities is continuous in $\tau$ if $\lambda_{\tau^n} \Rightarrow \lambda_\tau$ whenever $\tau^n \rightarrow \tau$ in $Q$.

As described in the Appendix B.2, the random matrix $W(\tilde{{\bm \theta}})$ is a function of the unobserved variables $(x_1 , \cdots , x_m)$ and the parameter $\tilde{{\bm \theta}}$. Then, it is defined on the probability space $(\mathbb{R}^m , \beta(\mathbb{R}^m) , \mathbb{P}_{(x_1, \cdots , x_m)})$. We denote by $\mathbb{G}_{\tilde{\mbox{\boldmath$\theta$}}}$ the distribution of the random matrix $W(\tilde{{\bm \theta}})$.

\begin{lemma} \label{conv_uniformly_weakly} Given the sequences  $\{\mbox{\boldmath$\psi$}^n: \mbox{\boldmath$\psi$}^n \in B_c(n,\mbox{\boldmath$\theta$}), ~ n \geq 1\}$ and $\{\mbox{\boldmath$\phi$}^n: \mbox{\boldmath$\phi$}^n \in R_c(n,\mbox{\boldmath$\theta$}), ~ n \geq 1\}$ and  $g: M^{m + 2(p-1)} \rightarrow \mathbb{R}$ a bounded continuous function, then

\[
\mathbb{E}_{\mbox{\boldmath$\psi$}^n} \left[ g \left(\frac{1}{n} J^n({\bm \phi}^n) \right) \right] = \int g \left(\frac{1}{n} J^n({\bm \phi}^n) \right) d \mathbb{P}_{\mbox{\boldmath$\psi$}^n} \rightarrow \int g \left(W(\tilde{{\bm \theta}}) \right) d \mathbb{P}_{(x_1, \cdots , x_m)} = \mathbb{E}_{(x_1, \cdots , x_m)} \left[ g \left( W(\tilde{{\bm \theta}}) \right) \right],
\] for every ${\bm \theta} \in \mathbb{R}^{m+2(p-1)}$. Moreover, the distribution  $\mathbb{G}_{\tilde{\mbox{\boldmath$\theta$}}}$ of $W(\tilde{{\bm \theta}})$ is continuous in $\tilde{{\bm \theta}}$.
\end{lemma}

\begin{proof}The fact that  $\mathbb{G}_{\tilde{\mbox{\boldmath$\theta$}}}$ is continuous follows from \cite{sweeting1980uniform}, Lemma 3.
It follows from Lemma \ref{conv_obs} that $\frac{1}{n} J^n({\bm \phi}^n)$ converges uniformly in probability to $W(\tilde{{\bm \theta}})$. As uniformly convergence in probability implies uniformly convergence in distribution (\cite{sweeting1980uniform}, Lemma 2), it follows from Lemma 1 in \cite{sweeting1980uniform} that

\begin{equation}
 \int g \left(\frac{1}{n} J^n({\bm \phi}^n) \right) d \mathbb{P}_{\mbox{\boldmath$\psi$}^n , (x_1, \cdots , x_m)}
\rightarrow \int g \left(W(\tilde{{\bm \theta}}) \right) d \mathbb{P}_{\tilde{\mbox{\boldmath$\theta$}} ,(x_1, \cdots , x_m)},
\end{equation} for any bounded continuous function $g$. As $J^n({\bm \phi}^n)$ depends only on ${\bf Y}^n$, we conclude that

\[
\int g \left(\frac{1}{n} J^n({\bm \phi}^n) \right) d \mathbb{P}_{\mbox{\boldmath$\psi$}^n }  = \int g \left(\frac{1}{n} J^n({\bm \phi}^n) \right) d \mathbb{P}_{\mbox{\boldmath$\psi$}^n , (x_1, \cdots , x_m)}.
\] The fact that $W(\tilde{{\bm \theta}})$ depends only on the unobservable variables $(x_1, \cdots , x_m)$ yields

\[
\int g \left(W(\tilde{{\bm \theta}}) \right) d \mathbb{P}_{(x_1, \cdots , x_m)} = \int g \left(W(\tilde{{\bm \theta}}) \right) d \mathbb{P}_{\tilde{\mbox{\boldmath$\theta$} },(x_1, \cdots , x_m)}.
\]  As by product we conclude the Lemma.
\end{proof}

Let ${\bf s} \in \mathbb{R}^{m+2(p-1)}$ be a vector and let $\{{\bm \theta}^n : {\bm \theta}^n \in B_c(n , {\bm \theta}), n \geq 1\}$ be a sequence of parameters.  We define $ \mbox{\boldmath$\psi$}^n = \mbox{\boldmath$\theta$}^n + \frac{1}{\sqrt{n}}{\bf s}$ a vector in $\mathbb{R}^{m+2(p-1)}$ such that $\mbox{\boldmath$\psi$}^n \rightarrow \mbox{\boldmath$\theta$}$ as $n \rightarrow \infty$. As $L^n(\mbox{\boldmath$\theta$}^n)$ is a smooth function , we may write
\begin{eqnarray} \nonumber
L^n (\mbox{\boldmath$\psi$}^n) &=& L^n (\mbox{\boldmath$\theta$}^n) + (\mbox{\boldmath$\psi$}^n - \mbox{\boldmath$\theta$}^n)^T U^n(\mbox{\boldmath$\theta$}^n) - \frac{1}{2} (\mbox{\boldmath$\psi$}^n - \mbox{\boldmath$\theta$}^n)^T J^n(\mbox{\boldmath$\phi$}^n) (\mbox{\boldmath$\psi$}^n - \mbox{\boldmath$\theta$}^n) \\ \label{decomp_likelihood}
&=& L^n (\mbox{\boldmath$\theta$}^n) + \frac{1}{\sqrt{n}}{\bf s}^T U^n(\mbox{\boldmath$\theta$}^n) - \frac{1}{2n} {\bf s}^T J^n(\mbox{\boldmath$\phi$}^n) {\bf s} ,
\end{eqnarray} where $\mbox{\boldmath$\phi$}^n =  (1-\delta^n) \mbox{\boldmath$\theta$}^n + \delta^n\mbox{\boldmath$\psi$}^n$, $0 < \delta^n < 1$ and $\delta^n$ is random. As $\delta^n$ is a function of the observed data ${\bf Y}^n$ and $0 < \delta^n  < 1$, we conclude that $\mbox{\boldmath$\phi$}^n \in R_c (n , {\bm \theta})$. As a consequence, we obtain that $\mbox{\boldmath$\phi$}^n \rightarrow \mbox{\boldmath$\theta$}$ as $n \rightarrow \infty$.

\begin{theorem} \label{Score_limit} Let $\{{\bm \theta}^n : {\bm \theta}^n \in B_c(n , {\bm \theta}), n \geq 1\}$ be a sequence of parameters and let $\{{\bm h}^n : {\bm h}^n \in R_c(n , {\bm \theta}), n \geq 1\}$ be a sequence of random vectors. Then, we have that

\[
\left(\frac{1}{\sqrt{n}}  U^n (\mbox{\boldmath$\theta$}^n) , \frac{1}{n} J^n ({\bm h}^n) \right)  ~ \Rightarrow_u ~ \left( H (\tilde{\mbox{\boldmath$\theta$}}), W(\tilde{\mbox{\boldmath$\theta$}})  \right)
\] where $H (\tilde{\mbox{\boldmath$\theta$}}) = ({\bm 0}_m^T , \left((\tilde{W}(\tilde{\mbox{\boldmath$\theta$}}))^{1/2} {\bf z}\right)^T)^T$ such that

\begin{itemize}
  \item ${\bm 0}_m $ is the null vector of dimension $m$;
  \item ${\bf z}$ is a standard normal random vector on $\mathbb{R}^{2(p-1)}$, independent of the random matrix $\tilde{W}(\tilde{\mbox{\boldmath$\theta$}})$.
\end{itemize}
Furthermore, the random matrix $\tilde{W}(\tilde{\mbox{\boldmath$\theta$}})$ is positive definite with probability one and it depends only on the bias components $\tilde{\mbox{\boldmath$\theta$}}$ and the unobserved variables $(x_1, \cdots , x_m)$.
\end{theorem}
\begin{proof}

Taking exponentials in equation (\ref{decomp_likelihood}) and rearranging gives

\begin{equation} \label{log_vero_eq}
  \exp\left[\frac{1}{2n} {\bf s}^T J^n(\mbox{\boldmath$\phi$}^n) {\bf s}  \right]   f_{{\bf Y}^n}({\bf y}^n ,\mbox{\boldmath$\psi$}^n)=  \exp \left[ \frac{1}{\sqrt{n}}{\bf s}^T U^n(\mbox{\boldmath$\theta$}^n)\right]  f_{{\bf Y}^n}({\bf y}^n ,\mbox{\boldmath$\theta$}^n).
\end{equation} Let $0 < \epsilon < 1$ and choose a positive constant $v$ such that $\mathbb{P}_{(x_1, \cdots , x_m)} [ \parallel W(\tilde{\mbox{\boldmath$\theta$}}) \parallel \geq v ] \leq \epsilon$. As a consequence of Lemma \ref{conv_uniformly_weakly}, we have that $ (1/n) J^n(\mbox{\boldmath$\phi$}^n)$ converges uniformly in distribution to $W(\tilde{\mbox{\boldmath$\theta$}})$  under the families of probabilities $\{\mathbb{P}_{\mbox{\boldmath$\psi^n$}} : n \geq 1 \}$ and  $\{ \mathbb{P}_{{\bm \theta}^n}: n \geq 1 \}$. Since $\{A \in M^{m+2(p-1)}: \parallel A \parallel < v \}$ is a $\mathbb{G}_{\mbox{\boldmath$\theta$}}$-continuity set, it follows from Lemmas 1 and 3 in \cite{sweeting1980uniform} that

\begin{equation} \label{conv_set_weak}
\mathbb{P}_{\mbox{\boldmath$\theta$}^n} \left[ \parallel \frac{1}{n} J^n(\mbox{\boldmath$\phi$}^n) \parallel < v  \right] \rightarrow \mathbb{P}_{(x_1, \cdots , x_m)} \left[ \parallel W (\tilde{\mbox{\boldmath$\theta$}}) \parallel < v  \right].
\end{equation}

Let $\mathbb{Q}_{\mbox{\boldmath$\theta$}^n}$ be the probability $\mathbb{P}_{\mbox{\boldmath$\theta$}^n}$ conditional on $\{\parallel \frac{1}{n} J^n(\mbox{\boldmath$\phi$}^n) \parallel < v \}$. In this case, the finite dimensional component of $\mathbb{Q}_{\mbox{\boldmath$\theta$}^n}$ has the following density

\[
q_{{\bf Y}^n} ({\bf y }^n , \mbox{\boldmath$\theta$}^n) =
\left\{
  \begin{array}{ll}
    \frac{f_{{\bf Y}^n} ({\bf y }^n , \mbox{\boldmath$\theta$}^n)}{\mathbb{P}_{\mbox{\boldmath$\theta$}^n} \left[ \parallel \frac{1}{n} J^n(\mbox{\boldmath$\phi$}^n) \parallel < v  \right]}, & \parallel \frac{1}{n} J^n(\mbox{\boldmath$\phi$}^n) \parallel < v  \\ \\
    0, & otherwise.
  \end{array}
\right.
\]

Let $g$ be a bounded function on $M^{m + 2(p-1)}$, continuous on $\parallel A \parallel < v$ and with $g(A)=0$ for every $\parallel A \parallel \geq v$. Let $\mathbb{E}_{\mbox{\boldmath$\theta$}^n}^\star$ denotes the expectation with respect to the probability $\mathbb{Q}_{\mbox{\boldmath$\theta$}^n}$ and $\mathbb{E}_{(x_1, \cdots , x_m)}^\star$ denotes the expectation with respect to the probability $\mathbb{P}_{(x_1, \cdots , x_m)}$ conditional on the set $\{ \parallel W(\tilde{{\bm \theta}}) \parallel < v\}$. Multiplying equation (\ref{log_vero_eq}) through by $g((1/n)J^n({\bm \phi}^n))$ and integrating with respect to the Lebesgue measure over the set $\{ \parallel (1/n)J^n({\bm \phi}^n) \parallel < v\}$ yields

\begin{eqnarray*}
\mathbb{E}_{\mbox{\boldmath$\theta$}^n}^\star \left[g \left(\frac{1}{n} J^n({\bm \phi}^n) \right) \exp \left( \frac{1}{\sqrt{n}}{\bf s}^T U^n(\mbox{\boldmath$\theta$}^n)\right)  \right] &=& \frac{ \mathbb{E}_{\mbox{\boldmath$\psi$}^n} \left[g \left((\frac{1}{n} J^n({\bm \phi}^n) \right) \exp\left(\frac{1}{2n} {\bf s}^T J^n(\mbox{\boldmath$\phi$}^n) {\bf s}  \right)  \right]}{\mathbb{P}_{\mbox{\boldmath$\theta$}^n} \left[ \parallel \frac{1}{n} J^n(\mbox{\boldmath$\phi$}^n) \parallel < v  \right]} \\ \\
& \rightarrow & \frac{\mathbb{E}_{(x_1, \cdots , x_m)} \left[ g(W(\tilde{\mbox{\boldmath$\theta$}})) \exp \left( \frac{1}{2} {\bf s}^T W(\tilde{\mbox{\boldmath$\theta$}}) {\bf s}  \right) \right]}{\mathbb{P}_{(x_1, \cdots , x_m)} \left[ \parallel W(\tilde{\mbox{\boldmath$\theta$}}) \parallel < v  \right]} \\ \\
&=& \mathbb{E}_{(x_1, \cdots , x_m)}^\star \left[g(W(\tilde{\mbox{\boldmath$\theta$}})) \exp \left( \frac{1}{2} {\bf s}^T W(\tilde{\mbox{\boldmath$\theta$}}) {\bf s}  \right) \right],
\end{eqnarray*} as a consequence of equation (\ref{conv_set_weak}), Lemma \ref{conv_uniformly_weakly} and the fact that $ g(A) \exp ( (1/2) {\bf s}^T A {\bf s}) $ is a bounded $\mathbb{G}_{{\bm \theta}}$-continuous function (see, \cite{Billingsley1968}, Theorem 5.2).

We decompose the vector ${\bf s}$ in two components ${\bf s}= ({\bf s}_1^T , {\bf s}_2^T)^T$ where ${\bf s}_1=(s_1, \cdots , s_m)^T$ and ${\bf s}_2=(s_{m+1} , \cdots , s_{m+2(p-1)})^T$. As every component of $W(\tilde{{\bm \theta}})$ associated with $\mu_{x_j}$ is null, we obtain that

\[
\mathbb{E}_{(x_1, \cdots , x_m)}^\star \left[g(W(\tilde{\mbox{\boldmath$\theta$}})) \exp \left( \frac{1}{2} {\bf s}^T W(\tilde{\mbox{\boldmath$\theta$}}) {\bf s}  \right) \right] = \mathbb{E}_{(x_1, \cdots , x_m)}^\star \left[g(W(\tilde{\mbox{\boldmath$\theta$}})) \exp \left({\bf s}^T_1 {\bf 0}_m + \frac{1}{2} {\bf s}^T_2 \tilde{W}(\tilde{\mbox{\boldmath$\theta$}}) {\bf s}_2  \right) \right] =
\]

\[
\mathbb{E}_{(x_1, \cdots , x_m)}^\star \left[g(W(\tilde{\mbox{\boldmath$\theta$}})) \exp \left( {\bf s}^T H(\tilde{\mbox{\boldmath$\theta$}})  \right)  \right],
\] such that  $H(\tilde{\mbox{\boldmath$\theta$}})=({\bf 0}_m, (\tilde{W}(\tilde{{\bm \theta}}))^{1/2} {\bf z})$ where ${\bf 0}_m$ is the null vector of dimension $m$ and ${\bf z}$ is a standard normal random vector on $\mathbb{R}^{2(p-1)}$, independent of the random matrix $\tilde{W}(\tilde{\mbox{\boldmath$\theta$}})$. By the uniqueness of the moment generating function and the weak compactness theorem, we conclude that

\[
\left(\frac{1}{n} J^n({\bm \phi}^n) , \frac{1}{\sqrt{n}} U^n ({\bm \theta}^n) \right) \Rightarrow \left( W(\tilde{{\bm \theta}}) , H(\tilde{\mbox{\boldmath$\theta$}}) \right) 1\!\!1_{ \{\parallel W(\tilde{{\bm \theta}}) \parallel < v  \}}
\] under the family $\{ \mathbb{Q}_{{\bm \theta}^n} : n \geq 1\}$ of probabilities. As $v$ is arbitrary, it follows that

\[
\left(\frac{1}{n} J^n({\bm \phi}^n) , \frac{1}{\sqrt{n}} U^n ({\bm \theta}^n) \right) \Rightarrow \left( W(\tilde{{\bm \theta}}) , H(\tilde{\mbox{\boldmath$\theta$}}) \right)
\] under the family $\{ \mathbb{P}_{{\bm \theta}^n} : n \geq 1\}$ of probabilities. By applying Lemma \ref{conv_obs}, we conclude that  $$\parallel \frac{1}{n} J^n({\bm \phi}^n) - \frac{1}{n} J^n({\bm h}^n) \parallel \rightarrow_u 0$$ uniformly in probability. Hence, we obtain that

\[
\left(\frac{1}{n} J^n({\bm h}^n) , \frac{1}{\sqrt{n}} U^n ({\bm \theta}^n) \right) \Rightarrow_u \left( W(\tilde{{\bm \theta}}) , H(\tilde{\mbox{\boldmath$\theta$}}) \right).
\]

\end{proof}

In the sequel, we will show the  asymptotic normality of the MLE regarding the bias parameters $\tilde{{\bm \theta}}$.  For every vector  ${\bf s}= ({\bf s}_1 , {\bf s}_2)
\in \mathbb{R}^{m+2(p-1)}$ such that ${\bf s}_1 =(s_1, \cdots ,s_m)$ and ${\bf s}_2=(s_{m+1} , \cdots , s_{m+2(p-1)})$, it follows from equation  (\ref{decomp_likelihood}) that

\begin{eqnarray*}
L^n (\mbox{\boldmath$\psi$}^n) - L^n (\mbox{\boldmath$\theta$}) &=&  \frac{1}{\sqrt{n}}{\bf s}^T U^n(\mbox{\boldmath$\theta$}) - \frac{1}{2n} {\bf s}^T J^n(\mbox{\boldmath$\phi$}^n) {\bf s} = \frac{1}{\sqrt{n}}{\bf s}^T_2 \tilde{U}^n (\mbox{\boldmath$\theta$}) -   \frac{1}{2n} {\bf s}^T_2 \tilde{J}^n(\mbox{\boldmath$\theta$}) {\bf s}_2 \\
&+&  \frac{1}{\sqrt{n}}{\bf s}^T_1 U^n_{\mu_{x_j}} (\mbox{\boldmath$\theta$}) -\frac{1}{2n} \sum_{i=1}^m s_i^2 J^n_{ii}(\mbox{\boldmath$\phi$}^n) - \frac{1}{2n} \sum_{i=1}^m \sum_{j=m+1}^{m+2(p-1)} s_i s_j J^n_{ij}(\mbox{\boldmath$\phi$}^n) \\
&-&  \frac{1}{2n} \sum_{i=m+1}^{m+2(p-1)} \sum_{j=1}^{m} s_i s_j J^n_{ij}(\mbox{\boldmath$\phi$}^n) + \frac{1}{2n} {\bf s}^T_2  \left[\tilde{J}^n(\mbox{\boldmath$\theta$}) - \tilde{J}^n(\mbox{\boldmath$\phi$}^n) \right] {\bf s}_2,
\end{eqnarray*} where ${\bm \theta} \in \mathbb{R}^{m+2(p-1)}$, $\mbox{\boldmath$\psi$}^n = \mbox{\boldmath$\theta$} + \frac{1}{\sqrt{n}}{\bf s}$  and $\mbox{\boldmath$\phi$}^n = (1-\delta^n) \mbox{\boldmath$\theta$} + \delta^n \mbox{\boldmath$\psi$}^n$, $0 < \delta^n < 1$ such that $\delta^n$ is random. As a consequence of Lemma \ref{conv_obs} and equation (\ref{null_inform_VV}), we arrive at the following lemma.

\begin{lemma} We have that

\begin{equation}  \label{cons_mle1}
\sup_{{\bf s} \in \mathbb{R}^{m+2(p-1)}} ~ \left[ L^n (\mbox{\boldmath$\psi$}^n) - L^n (\mbox{\boldmath$\theta$})\right] = \sup_{{\bf s}_2 \in \mathbb{R}^{2(p-1)}} \left[ \frac{1}{\sqrt{n}}{\bf s}^T_2 \tilde{U}^n (\mbox{\boldmath$\theta$}) -   \frac{1}{2n} {\bf s}^T_2 \tilde{J}^n(\mbox{\boldmath$\theta$}) {\bf s}_2  \right] + o_p (1),
\end{equation} for ${\bm \theta} \in \mathbb{R}^{m+2(p-1)}$.
\end{lemma}
This Lemma is crucial to understand the behavior of the likelihood function with respect to the true value parameters $(\mu_{x_1}, \cdots , \mu_{x_m})$. For $n$ sufficiently large, the impact of the true values vanish. Moreover, the maximum with respect to ${\bf s}_2$ of the right side of equation (\ref{cons_mle1}) satisfies \begin{equation} \label{asymp_normality1}  \frac{1}{n} \tilde{J}^n(\mbox{\boldmath$\theta$}) \hat{{\bf s}}_2 = \frac{1}{\sqrt{n}} \tilde{U}^n (\mbox{\boldmath$\theta$}) + o_p(1).\end{equation}
By applying equation (\ref{cons_mle1}), for $n$ sufficiently large, we conclude that $({\bf s}_1^T, \hat{{\bf s}}_2^T)^T$  corresponds closely to the value of $\hat{{\bm s}}$ that maximazes $L^n (\mbox{\boldmath$\theta$} + n^{-1/2} ({\bf s}_1^T, {\bf s}_2^T)^T)$ independent of the vector ${\bf s}_1$.

The maximum of $L^n (\mbox{\boldmath$\theta$} + n^{-1/2} {\bf s})$, the MLE $\hat{{{\bm \theta}}}^n$, is given by
\[
\hat{{{\bm \theta}}}^n = \tilde{{\bm \theta}} + \frac{1}{\sqrt{n}} \hat{{\bf s}},  \quad \text{which~ gives} \]
\[
\hat{\tilde{{\bm \theta}}}^n = \tilde{{\bm \theta}} + \frac{1}{\sqrt{n}} \hat{{\bf s}}_2  \quad \text{and} \quad \sqrt{n} \left(  \hat{\tilde{{\bm \theta}}}^n - \tilde{{\bm \theta}} \right) = \hat{{\bf s}}_2,
\] where $\hat{\tilde{{\bm \theta}}}^n$ corresponds to the MLE of the bias parameters $\tilde{{\bm \theta}}$. As a consequence, we conclude that

\begin{equation} \label{asymp_normality2}
\left( \frac{1}{n} \tilde{J}^n(\mbox{\boldmath$\theta$}) \right) \sqrt{n} \left(  \hat{\tilde{{\bm \theta}}}^n - \tilde{{\bm \theta}} \right) =  \frac{1}{n} \tilde{J}^n(\mbox{\boldmath$\theta$}) \hat{{\bf s}}_2.
\end{equation} Summing up the results obtained from equations (\ref{asymp_normality1}) and (\ref{asymp_normality2}), we arrive at the following Theorem.

\begin{theorem} \label{asymp_normality_2} The MLE $\hat{\tilde{\mbox{\boldmath$\theta$}}}^n$ of $\tilde{\mbox{\boldmath$\theta$}}$ satisfies

\[
\frac{1}{\sqrt{n}} \tilde{U}^n (\mbox{\boldmath$\theta$}) - \left( \frac{1}{n} \tilde{J}^n(\mbox{\boldmath$\theta$}) \right) \sqrt{n} \left(  \hat{\tilde{{\bm \theta}}}^n - \tilde{{\bm \theta}} \right) \rightarrow_u 0,
\] uniformly in probability, where $\mbox{\boldmath$\theta$} = (\mu_{x_1} , \cdots , \mu_{x_m} , \tilde{\mbox{\boldmath$\theta$}}) \in \mathbb{R}^{m+2(p-1)}$.

\end{theorem}

As a consequence Theorems \ref{Score_limit} and \ref{asymp_normality_2} and the continuous mapping theorem, we conclude that

\begin{equation} \label{TesteWald}
\left(\left[\frac{1}{n} \tilde{J}^n ({\bm \theta})  \right]^{1/2} \sqrt{n} \left(  \hat{\tilde{{\bm \theta}}}^n - \tilde{{\bm \theta}} \right) , \frac{1}{n} \tilde{J}^n ({\bm \theta})    \right) \Rightarrow_u \left( {\bf z} , \tilde{W} (\tilde{{\bm \theta}}) \right).
\end{equation} Equation (\ref{TesteWald}) and the continuous mapping Theorem yield

\begin{equation}  \label{fan_MLE}
\sqrt{n} \left( \hat{\tilde{{\bm \theta}}}^n - \tilde{{\bm \theta}} \right) \Rightarrow_u \left[\tilde{W} (\tilde{{\bm \theta}})  \right]^{-1/2} ~ {\bf z}.
\end{equation} From equation (\ref{fan_MLE}), we know that the asymptotic distribution of the MLE regarding the bias parameters $\tilde{{\bm \theta}}$ is not normal, because the matrix $\tilde{W} (\tilde{{\bm \theta}})$ is random.

\begin{corollary} \label{Consi_MLE} The MLE $\hat{\tilde{{\bm \theta}}}^n$ related to the bias parameters satisfies

\[
\left( \hat{\tilde{{\bm \theta}}}^n - \tilde{{\bm \theta}} \right) \rightarrow_u 0.
\]

\end{corollary}

In the sequel, we will derive the usual Wald statistics to perform hypothesis testing about the bias parameters. In order to do it, it is necessary to derive Theorem \ref{Score_limit} with ${\bm h}^n  = \hat{{\bm \theta}}^n$. By applying Prohorov's Theorem , we know that the sequence $\{\sqrt{n} \left( \hat{\tilde{{\bm \theta}}}^n - \tilde{{\bm \theta}} \right) : n \geq 1\}$ is uniformly tight. Then, for each $\epsilon > 0$, there exists a constant $c > 0$ such that

\[
\mathbb{P}_{ {\bm \theta}^n , (x_1, \cdots , x_m)} \left[ \big| \sqrt{n} \left( \hat{\tilde{{\bm \theta}}}^n - \tilde{{\bm \theta}} \right) \big| > c  \right] < \epsilon, \quad n \geq 1.
\] As a consequence, with probability tending to one, $\hat{\tilde{{\bm \theta}}}^n \in R_c (n , \tilde{{\bm \theta}})$.

\begin{lemma} Let $\{{\bm \theta}^n : {\bm \theta}^n \in B_c(n , {\bm \theta}), n \geq 1\}$ be a sequence of parameters. Then, we have that
\[
\left(\frac{1}{\sqrt{n}}  U^n (\mbox{\boldmath$\theta$}^n) , \frac{1}{n} \tilde{J}^n (\hat{{\bm \theta}}^n) \right)  ~ \Rightarrow_u ~ \left( H (\tilde{\mbox{\boldmath$\theta$}}), \tilde{W}(\tilde{\mbox{\boldmath$\theta$}})  \right)
\]
\end{lemma}

\begin{proof} For any $(\mu_{x_1} , \cdots , \mu_{x_m}) \in \mathbb{R}^m$, consider $ {\bf h}^n = (\mu_{x_1} , \cdots , \mu_{x_m} , \tilde{\mbox{\boldmath$\theta$}})$  and $A^n = \{ \sqrt{n} \left( {\bf h}^n - {\bm \theta} \right) \big| \leq c\}$. It follows from Prohorov's Theorem that $1\!\!1_{A^n} \rightarrow_u 0$, for every positive constant $c$. Hence, as a consequence of Theorem \ref{Score_limit}

\[
\left(\frac{1}{\sqrt{n}}  U^n (\mbox{\boldmath$\theta$}^n) , \frac{1}{n} \tilde{J}^n ({\bm h}^n) \right)  ~ \Rightarrow_u ~ \left( H (\tilde{\mbox{\boldmath$\theta$}}), \tilde{W}(\tilde{\mbox{\boldmath$\theta$}})  \right).
\] As $\tilde{W}$ does not depend on $(\mu_{x_1} , \cdots , \mu_{x_m})$, the arrive at the result of the Lemma.

\end{proof}

 Thus, we arrive at the following Corollaries.

\begin{corollary} \label{asymp_normality_3} Conditional on $\tilde{J}^n (\hat{{\bm \theta}}^n)$, the asymptotic distribution of $\sqrt{n} \left( \hat{\tilde{{\bm \theta}}}^n - \tilde{{\bm \theta}} \right)$  is given by $$N_{2(p-1)} \left[ {\bm 0}_{2(p-1)} , \left( \frac{1}{n} \tilde{J}^n (\hat{{\bm \theta}}^n) \right)^{-1} \right].$$
\end{corollary} Applying again equation (\ref{TesteWald}) and continuous mapping Theorem  we arrive at the Wald statistics.

\begin{corollary} \label{wald_test}
We have that

$$  \left(  \hat{\tilde{{\bm \theta}}}^n - \tilde{{\bm \theta}} \right)^T \left[\tilde{J}^n (\hat{{\bm \theta}}^n)  \right] \left(  \hat{\tilde{{\bm \theta}}}^n - \tilde{{\bm \theta}} \right) \Rightarrow_u  {\bf z}^T {\bf z}. $$
\end{corollary}

\section{Equivalence among Participant laboratories}
In this section, we will propose multiple hypothesis testing to assess the equivalence among the laboratories measurements with respect to the reference laboratory. Initially, we will test for the equivalence of all laboratories with respect to the reference laboratory,
\begin{equation} \label{eaal} H_0:\alpha_2=\cdots=\alpha_p=0 \quad \text{and}  \quad \beta_2=\cdots=\beta_p=1 .\end{equation} To test hypothesis (\ref{eaal}), we may apply the  Wald statistics as established in Corollary \ref{wald_test}, i.e.

	\begin{equation}\label{eqWald11}
	Q_w=\left(  \hat{\tilde{{\bm \theta}}}^n - \tilde{{\bm \theta}}_0 \right)^T \left[\tilde{J}^n (\hat{{\bm \theta}}^n)  \right] \left(  \hat{\tilde{{\bm \theta}}}^n - \tilde{{\bm \theta}}_0 \right),
	\end{equation}
with $\tilde{{\bm \theta}}_0=(\bm{0}_{(p-1)}^T,\bm{1}_{(p-1)}^T)^T$ for hypothesis defined in (\ref{eaal}). Under the conditions established in Corollary \ref{wald_test}, $Q_w$ as indicated above has an asymptotic $ \chi^2_{2(p-1)} $ distribution.

In the sequel, we consider tests of the composite hypothesis

\begin{equation} \label{composite_h}
H_{0h} : h (\tilde{\mbox{\boldmath$\theta$}}) = 0 ,
\end{equation} where $h : \mathbb{R}^{2(p-1)} \rightarrow \mathbb{R}^{r}$ is a vector-valued function such that the derivative matrix $H(\tilde{\mbox{\boldmath$\theta$}}) = (\partial / \partial \tilde{\mbox{\boldmath$\theta$}}) h (\tilde{\mbox{\boldmath$\theta$}})^T$ is continuous in $\tilde{\mbox{\boldmath$\theta$}}$ and the $Rank (H(\tilde{\mbox{\boldmath$\theta$}})) = r$. In order to develop these composite tests, consider the Taylor expansion

\[
h(\tilde{\mbox{\boldmath$\theta$}} + n^{-1/2} {\bf u}) = h (\tilde{\mbox{\boldmath$\theta$}}) + n^{-1/2} H^T(\tilde{\mbox{\boldmath$\theta$}}^\star) {\bf u} = h (\tilde{\mbox{\boldmath$\theta$}}) +n^{-1/2} H^T(\tilde{\mbox{\boldmath$\theta$}})  {\bf u} + n^{-1/2} \left[ H^T(\tilde{\mbox{\boldmath$\theta$}}^\star)-H^T(\tilde{\mbox{\boldmath$\theta$}}) \right]{\bf u},
\] where $\tilde{\mbox{\boldmath$\theta$}}^\star = \tilde{\mbox{\boldmath$\theta$}} + n^{-1/2} \gamma {\bf u}$, $0 < \gamma < 1$ and ${\bf u}\in  \mathbb{R}^{2(p-1)}$. By applying the assumption on the continuity of $H(\tilde{\mbox{\boldmath$\theta$}})$, we arrive at the following expression

\[
\sqrt{n} \left[ h(\tilde{\mbox{\boldmath$\theta$}} + n^{-1/2}{\bf u}) - h (\tilde{\mbox{\boldmath$\theta$}}) \right] = H^T(\tilde{\mbox{\boldmath$\theta$}})  {\bf u} + n^{-1/2} \left[ H^T(\tilde{\mbox{\boldmath$\theta$}}^\star)-H^T(\tilde{\mbox{\boldmath$\theta$}}) \right]{\bf u}.
\] Letting ${\bf u}= \sqrt{n} ( \hat{\tilde{\mbox{\boldmath$\theta$}}}^n - \tilde{\mbox{\boldmath$\theta$}})$, we obtain

\medskip
\begin{equation} \label{asy_comp}
\sqrt{n} \left[ h(\hat{\tilde{\mbox{\boldmath$\theta$}}}^n) - h (\tilde{\mbox{\boldmath$\theta$}}) \right] - H^T(\tilde{\mbox{\boldmath$\theta$}}) \sqrt{n} ( \hat{\tilde{\mbox{\boldmath$\theta$}}}^n - \tilde{\mbox{\boldmath$\theta$}}) = o_p(1) .
\end{equation}

\begin{theorem} \label{testecomposto} The compound Wald statistic

\begin{equation} \label{eqWald22}
Q_w = \left[ h(\hat{\tilde{\mbox{\boldmath$\theta$}}}^n) - h (\tilde{\mbox{\boldmath$\theta$}}) \right]^T \left[ H^T (\hat{\tilde{\mbox{\boldmath$\theta$}}}^n)\left( \tilde{J}^n (\hat{{\bm \theta}}^n) \right)^{-1} H(\hat{\tilde{\mbox{\boldmath$\theta$}}}^n) \right]^{-1} \left[ h(\hat{\tilde{\mbox{\boldmath$\theta$}}}^n) - h (\tilde{\mbox{\boldmath$\theta$}}) \right] \Rightarrow_u {\bf z}_r^T {\bf z}_r ,
\end{equation} has an asymptotic $ \chi^2_{r} $ distribution.

\end{theorem}

\begin{proof}
By applying equation \ref{TesteWald}, we have that

\[
\left( H^T(\tilde{\mbox{\boldmath$\theta$}}) \sqrt{n} \left(  \hat{\tilde{{\bm \theta}}}^n - \tilde{{\bm \theta}} \right) , \frac{1}{n} \tilde{J}^n ({\bm \theta})    \right) \Rightarrow_u \left( H^T(\tilde{\mbox{\boldmath$\theta$}}) \left( \tilde{W} ( \tilde{{\bm \theta}}) \right)^{-1/2} {\bf z} , \tilde{W} (\tilde{{\bm \theta}}) \right).
\] As the distribution of ${\bf z}$ is independent of $\tilde{W} (\tilde{{\bm \theta}})$, we obtain that

\[
\mathbb{P}_{(x_1, \cdots , x_m)} \left[ H^T(\tilde{\mbox{\boldmath$\theta$}}) \left( \tilde{W} ( \tilde{{\bm \theta}}) \right)^{-1/2} {\bf z} \in  D \mid \tilde{W} (\tilde{{\bm \theta}})  \right] = \mathbb{P}_{(x_1, \cdots , x_m)} \left[ \left[ H^T (\tilde{\mbox{\boldmath$\theta$}})\left( \tilde{W} (\tilde{{\bm \theta}})  \right)^{-1} H(\tilde{\mbox{\boldmath$\theta$}}) \right]^{1/2} {\bf z}_r \in  D \mid \tilde{W} (\tilde{{\bm \theta}})  \right],
\] for every $D \in \beta(\mathbb{R}^r)$, where ${\bf z}_r$ is a standard normal random vector on $\mathbb{R}^{r}$, independent of the random matrix $\tilde{W} (\tilde{{\bm \theta}})$. Then, the random vectors

\[
 H^T(\tilde{\mbox{\boldmath$\theta$}}) \left( \tilde{W} ( \tilde{{\bm \theta}}) \right)^{-1/2} {\bf z}  \quad \text{and} \quad \left[ H^T (\tilde{\mbox{\boldmath$\theta$}})\left( \tilde{W} (\tilde{{\bm \theta}})  \right)^{-1} H(\tilde{\mbox{\boldmath$\theta$}}) \right]^{1/2} {\bf z}_r
\] have the same distribution. As a consequence, we obtain that

\[
\left( H^T(\tilde{\mbox{\boldmath$\theta$}}) \sqrt{n} \left(  \hat{\tilde{{\bm \theta}}}^n - \tilde{{\bm \theta}} \right) , \frac{1}{n} \tilde{J}^n ({\bm \theta})    \right) \Rightarrow_u \left( \left[ H^T (\tilde{\mbox{\boldmath$\theta$}})\left( \tilde{W} (\tilde{{\bm \theta}})  \right)^{-1} H(\tilde{\mbox{\boldmath$\theta$}}) \right]^{1/2} {\bf z}_r , \tilde{W} (\tilde{{\bm \theta}}) \right).
\] By applying Corollary \ref{Consi_MLE} and Equation (\ref{asy_comp}), we conclude that

\[
\left[ h(\hat{\tilde{\mbox{\boldmath$\theta$}}}^n) - h (\tilde{\mbox{\boldmath$\theta$}}) \right]^T \left[ H^T (\hat{\tilde{\mbox{\boldmath$\theta$}}}^n)\left( \tilde{J}^n (\hat{{\bm \theta}}^n) \right)^{-1} H(\hat{\tilde{\mbox{\boldmath$\theta$}}}^n) \right]^{-1} \left[ h(\hat{\tilde{\mbox{\boldmath$\theta$}}}^n) - h (\tilde{\mbox{\boldmath$\theta$}}) \right] \Rightarrow_u {\bf z}_r^T {\bf z}_r .
\]

\end{proof}

If the null hypothesis  (\ref{eaal}) is rejected,  the multiple test is performed,
\begin{equation} \label{ebal} H_{0i}:\alpha_i=0 \quad \text{and} \quad\beta_i=1, \quad i=2, \cdots , p.
\end{equation}
Let $\left( \tilde{J}^n (\hat{{\bm \theta}}^n) \right)^{-1} =(v_{\theta_i
\theta_j})$, i.e., $v_{\theta_i \theta_j}$ is representing the
 $ijth$ element of the matrix $\left( \tilde{J}^n (\hat{{\bm \theta}}^n) \right)^{-1}$, then $Q_w$ for the hypothesis (\ref{ebal}) can be written as
\begin{equation} \label{testeind}Q_{wi}=\frac{(\hat{\tilde{\beta}}_{i}-1)^2 v_{\alpha_i
\alpha_i}-2\hat{\tilde{\alpha}}_i (\hat{\tilde{\beta}}_{i}-1)
v_{\alpha_i\beta_i}+\hat{\tilde{\alpha}}_i^2
v_{\beta_i\beta_i}}{v_{\alpha_i\alpha_i}
v_{\beta_i\beta_i}-v_{\alpha_i\beta_i}^2}.\end{equation}

As we are considering multiple test, it is important to control the type 1 error probability. For controlling the familywise error, it can be considered for example,  the Simes-Hockberg procedure (\cite{hochberg1988sharper}). To provide a graphical analysis of the performance of the laboratories measurements with respect to the measurements of the reference laboratory, the result obtained in
Theorem \ref{testecomposto} can be used to obtain the confidence regions. Next, we present a simulation study considering the tests given in (\ref{eqWald11}) and (\ref{eqWald22}).

\section{Simulation}
In this section we perform a simulation study to compare the behavior of the Wald test statistics developed in the previous section for different number of replicas, parameter values and nominal
levels of the test. Considering the model defined in (\ref{mod}) and (\ref{mod1}), with $p=5$ (number of participant laboratories) and $m=5$ (number of different engine rotation
values), it was generated 10000 samples with 3, 7, 15 and 30 replicas. The parameters of the true unobserved value of the item under testing at the $jth$ point ($x_j$, $j=1,\cdots,m$) was assumed to be:    $\mu_{x_1}=10,~~\mu_{x_2}=20,~~\mu_{x_3}=30,~~\mu_{x_4}=40$ and $\mu_{x_5}=50$, for the mean values and
$\sigma_{x_1}= 0.24,~~\sigma_{x_2}= 0.31,~~\sigma_{x_3}=0.38,~~\sigma_{x_4}= 0.45$ and $\sigma_{x_5}= 0.52$, for the standard deviations. It was considered three sets of parameter values for the standard deviation related to the measurement error of each laboratory ($i, i=1,\cdots,p$) at the $jth$ engine rotation value.

	\begin{enumerate}
		\item $ \sigma_{ij}^a:~~\sigma_{i1}=0.1 ,~~\sigma_{i2}= 0.2,~~ \sigma_{i3}= 0.3,~~ \sigma_{i4}= 0.4 ,~~\sigma_{i5}= 0.5 ; $
		\item $ \sigma_{ij}^b:~~\sigma_{i1}=0.2 ,~~\sigma_{i2}= 0.4,~~ \sigma_{i3}= 0.6,~~ \sigma_{i4}= 0.8 ,~~\sigma_{i5}= 1.0 ; $
		\item $ \sigma_{ij}^c:~~\sigma_{i1}=0.3 ,~~\sigma_{i2}= 0.6,~~ \sigma_{i3}= 0.9,~~ \sigma_{i4}= 1.2 ,~~\sigma_{i5}= 1.5  .$
	\end{enumerate}

Moreover, it was considered $\alpha=1\%$, $\alpha=5\%$ and $\alpha=10\%$ for the nominal significance levels. The routines were implemented in \cite{R2016}.

First, we consider the test for the equivalence of all laboratories with respect to the reference laboratory:
\[ H_0:\alpha_2=\cdots=\alpha_5=0 \quad \text{and}  \quad \beta_2=\cdots=\beta_5=1 .\]

It was obtained the empirical significance levels considering the test obtained in (\ref{eqWald11}). The results are summarized in Table \ref{yearlyprivatizatio01311212030320202n}.

\normalsize
\begin{table}[H]
		\centering
		\captionsetup{justification=centering}
		\caption{ Empirical sizes for the Wald test statistics for the test $ H_0: \alpha_2=\dots=\alpha_5=0,\beta_2=\dots=\beta_5=1 $}
		\label{yearlyprivatizatio01311212030320202n}
		{\normalsize
			\begin{tabular}{lccc ccc ccc}
				\hline
				\multicolumn{1}{c}{} & \multicolumn{3}{c}{$ \sigma_{ij}^a $} &
				\multicolumn{3}{c}{$ \sigma_{ij}^b $} & \multicolumn{3}{c}{$ \sigma_{ij}^c $} \\
				\cmidrule(l){2-4} \cmidrule(l){5-7} \cmidrule(l){8-10}
				$ n_i $ & 1\%  & 5\%  & 10\%  & 1\%  & 5\%  & 10\%   & 1\%  & 5\%  & 10\%  \\
				\hline
				3 			&  0.012 & 0.059 & 0.114
				& 0.023 & 0.084 & 0.15
				& 0.043 & 0.126 & 0.202
				\\
				7 			&   0.011 & 0.053 & 0.106
				& 0.015 & 0.065 & 0.127
				& 0.019 & 0.076 & 0.140
				\\
				15			& 0.011 & 0.053 & 0.102
				& 0.011 & 0.058 & 0.109
				& 0.017 & 0.068 & 0.124
				\\
				30			& 0.010 & 0.053 & 0.107
				& 0.011 & 0.053 & 0.102
				&  0.012 & 0.056 & 0.110
				\\
				\hline
			\end{tabular}
		}
	\end{table}
It can be noticed that as the number of replicas ($n_i$) increase, the empirical sizes approaches the nominal sizes. Also,
considering the first set of parameter values for the standard deviation of the measurement error of the laboratories ($\sigma_{ij}^a$) the nominal and empirical values
are close even for small number of replicas, however as these standard deviations increase ($\sigma_{ij}^b$ and $\sigma_{ij}^c$) we need a larger
number of replicas.

Next,
without loss of generality we consider the second laboratory to test for the equivalence of a laboratory with respect to the reference laboratory:
\[ H_0:\alpha_2=0 \quad \text{and}  \quad \beta_2=1 .\]

It was obtained the empirical significance levels considering the test obtained in (\ref{testeind}). The results are summarized in Table \ref{yearlyprivatizatio0131na} and it reaches the same conclusions as given above for the equivalence of all laboratories with respect to the reference laboratory.

\begin{table}[H]
		\centering
			\captionsetup{justification=centering}
		\caption{ Empirical sizes for the Wald test statistics for the test \\$ H_0: \alpha_2=0,~\beta_2=1 $ }
		\label{yearlyprivatizatio0131na}
		{
			\begin{tabular}{lccc ccc ccc }
				\hline
				\multicolumn{1}{c}{} & \multicolumn{3}{c}{\textbf{ $ \sigma_{ij}^a $}} & \multicolumn{3}{c}{\textbf{ $ \sigma_{ij}^b $}} & \multicolumn{3}{c}{\textbf{ $ \sigma_{ij}^c $}} \\
				\cmidrule(l){2-4} \cmidrule(l){5-7} \cmidrule(l){8-10}
				$ n_i $ & 1\%  & 5\%  & 10\%  & 1\%  & 5\%  & 10\%  & 1\%  & 5\%  & 10\%   \\
				\hline
				3 & 0.016 & 0.065 & 0.126
				& 0.024 & 0.081 & 0.147
				& 0.035 & 0.114 & 0.189
				 \\
				7 & 0.010 & 0.051 & 0.101
				& 0.017 & 0.070 & 0.129
				& 0.023 & 0.088 & 0.151
				\\
				15 & 0.010 & 0.051 & 0.101
				& 0.013 & 0.061 & 0.113
				& 0.016 & 0.068 & 0.126			
				\\
				30 & 0.008 & 0.048 & 0.102
				& 0.012 & 0.053 & 0.102
				& 0.013 & 0.063 & 0.120
				\\
				\hline
			\end{tabular}
		}
	\end{table}
Furthermore, to simulate the power of the test for the equivalence of all laboratories with respect to the reference laboratory,
it was considered a gradual distance from the null hypothesis for the second and forth laboratories and obtained the percentages of the observed values
of the test statistics which were greater than the $95^{th}$ quantile of the Chi-squared distribution with $8$ degree of freedom.

	Figure
\ref{poderTodos}
shows the power of the test
for different number of replicas ($n_i=3, 7, 15$ and $30$) with the parameter of the standard deviation of the measurement error of each laboratory at the $jth$ engine rotation value given by $ \sigma_{ij}^a $ and $ \sigma_{ij}^b $, respectively.

Notice that in both figures as the number of replicas increase the power of the test increases. Figure \ref{fig:cookPerturba.1Ap.M2058040320202562618} shows the power of the test as the standard deviation of the measurement error of the laboratories increases from $ \sigma_{ij}^a $ to $ \sigma_{ij}^b $ for fixed number of replicas .
In all cases the power of the test under $ \sigma_{ij}^a $ is greater than under $ \sigma_{ij}^b $. Another point to observe is the fact that the distance between the two curves (power under $ \sigma_{ij}^a $  and power under $ \sigma_{ij}^b $)
diminishes as the number of replicas increase.

Next, without loss of generality, we consider the test given in (\ref{ebal}) for the second laboratory. Figure \ref{fig:cookPerturba.1Ap.Multi.soUltimo.36142.only106.ld.52221elipse125620303032020618} shows the power of the test when
the standard deviation of the measurement error of the laboratories are given by $ \sigma_{ij}^a $ and $ \sigma_{ij}^b $ for different number of replicas.
As the number of replicas increase the power of the test increase, in addition when the standard deviation of the measurement error of the laboratories increase, the power decrease.

	\begin{figure}[H]
		\centering
		\captionsetup{justification=centering}
		\caption{ Simulated power for the Wald test statistics with $ \sigma_{ij}^a $ and $ \sigma_{ij}^b $ for the test
		$ H_0: \alpha_2=\dots=\alpha_5=0,\beta_2=\dots=\beta_5=1. $}
		\begin{tabular}{@{}cc@{}}
		\label{poderTodos}
			\includegraphics[width=.5\textwidth]{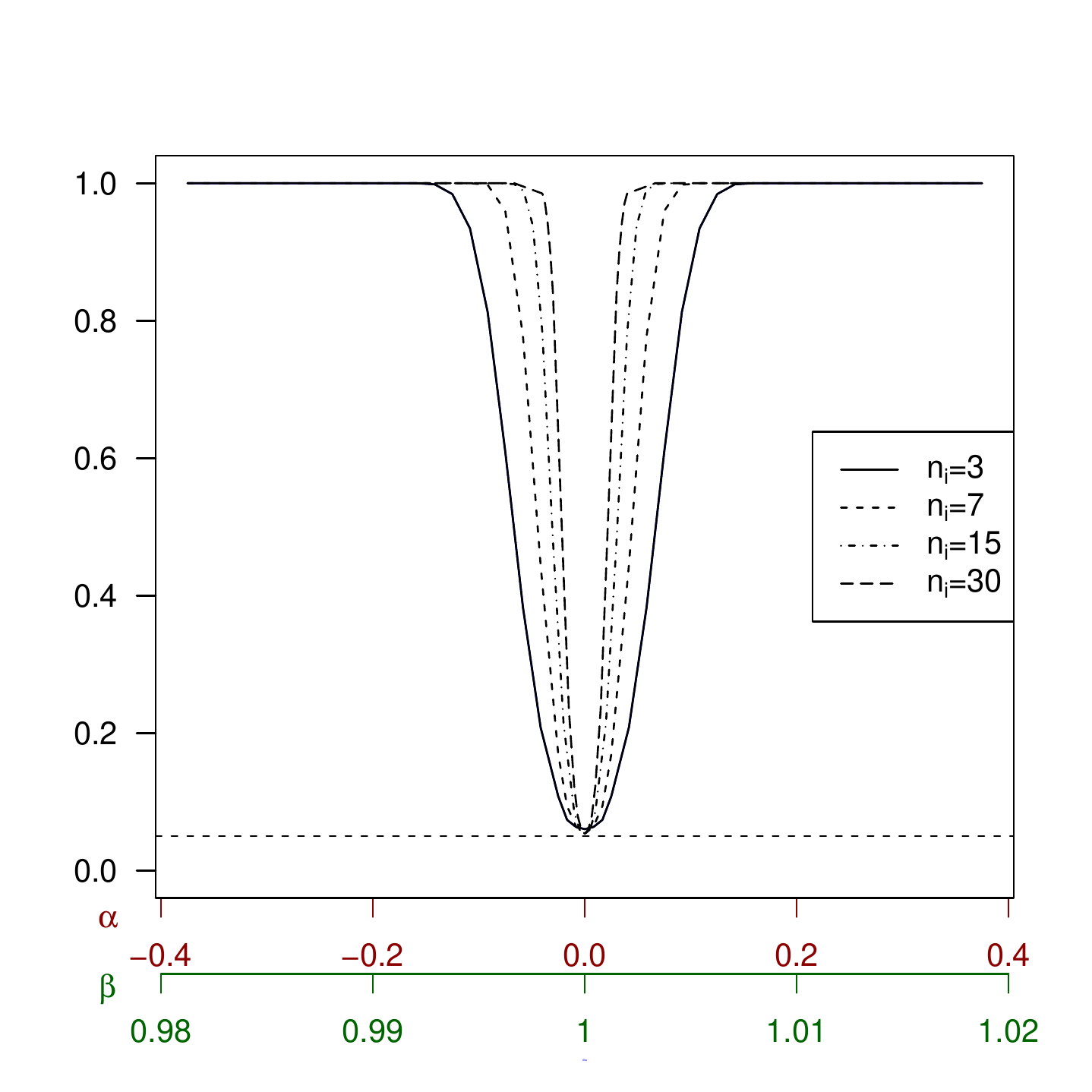} &	\includegraphics[width=.5\textwidth]{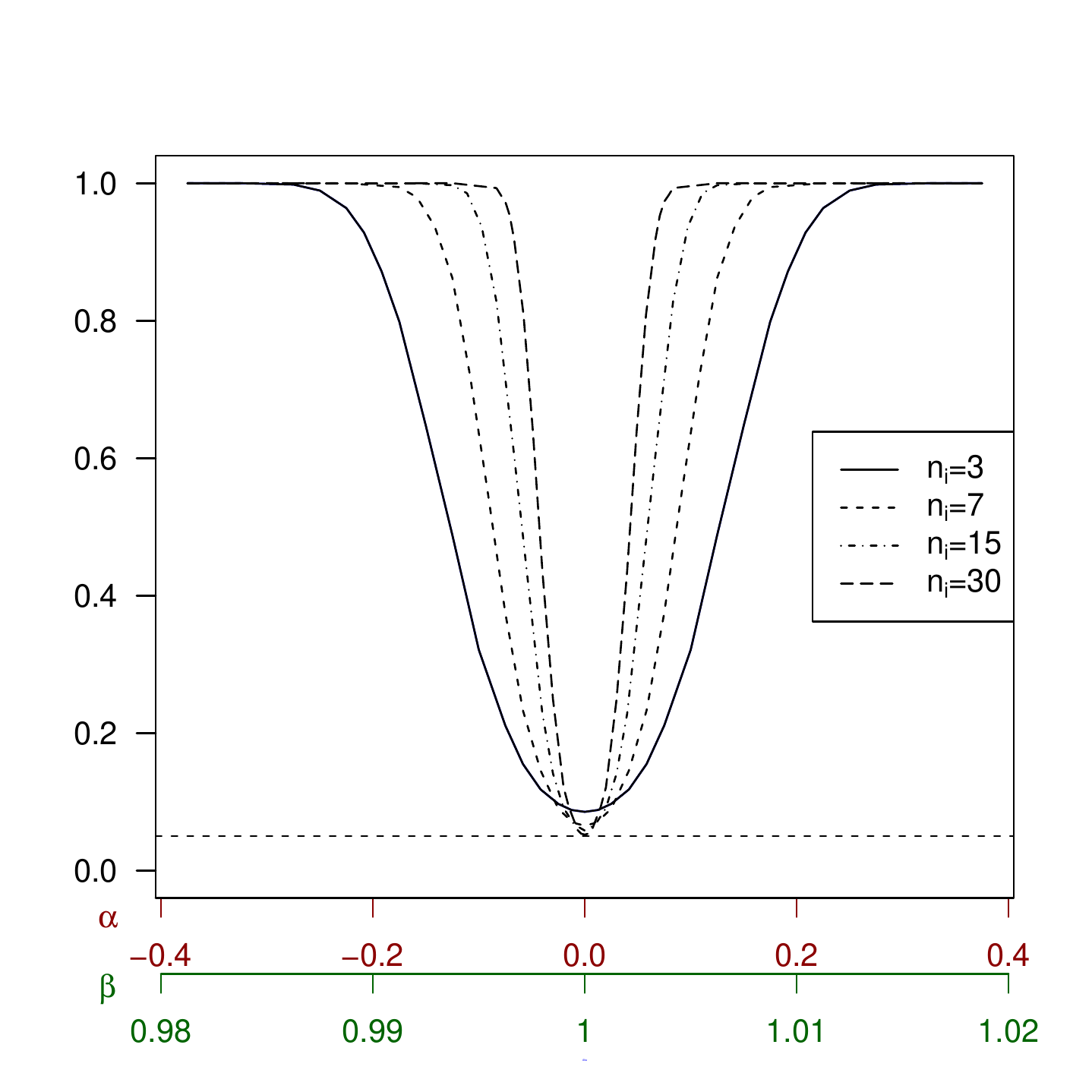}
		\end{tabular}
	\end{figure}

\begin{figure}[H]
		\centering
		\captionsetup{justification=centering}
		\caption{ Simulated power for the Wald test statistics with $ \sigma_{ij}^a $ and $ \sigma_{ij}^b $ for the test $ H_0: \alpha_2=\dots=\alpha_5=0,\beta_2=\dots=\beta_5=1. $   }
		\begin{tabular}{@{}cc@{}}
			\includegraphics[width=.45\textwidth]{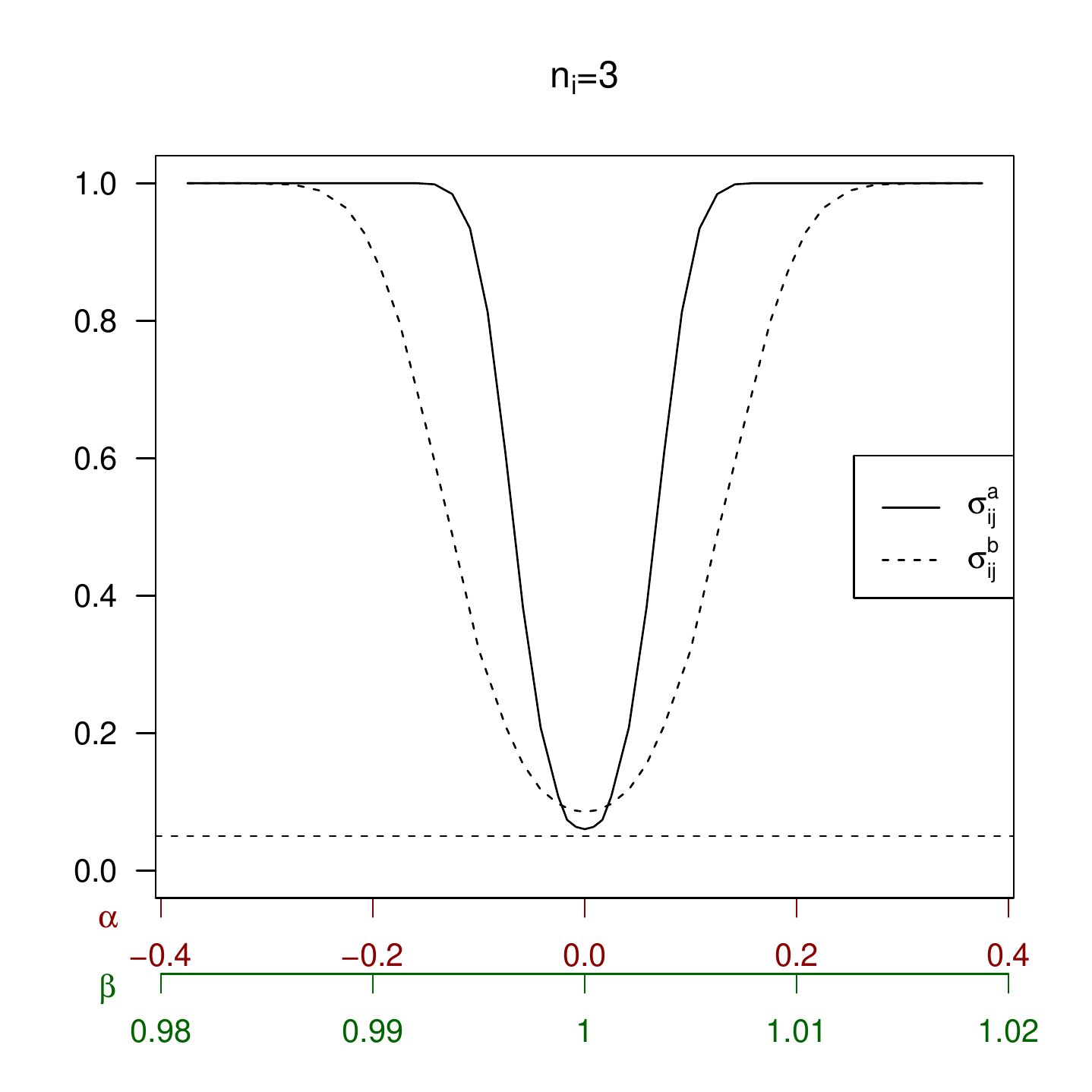} & \includegraphics[width=.45\textwidth]{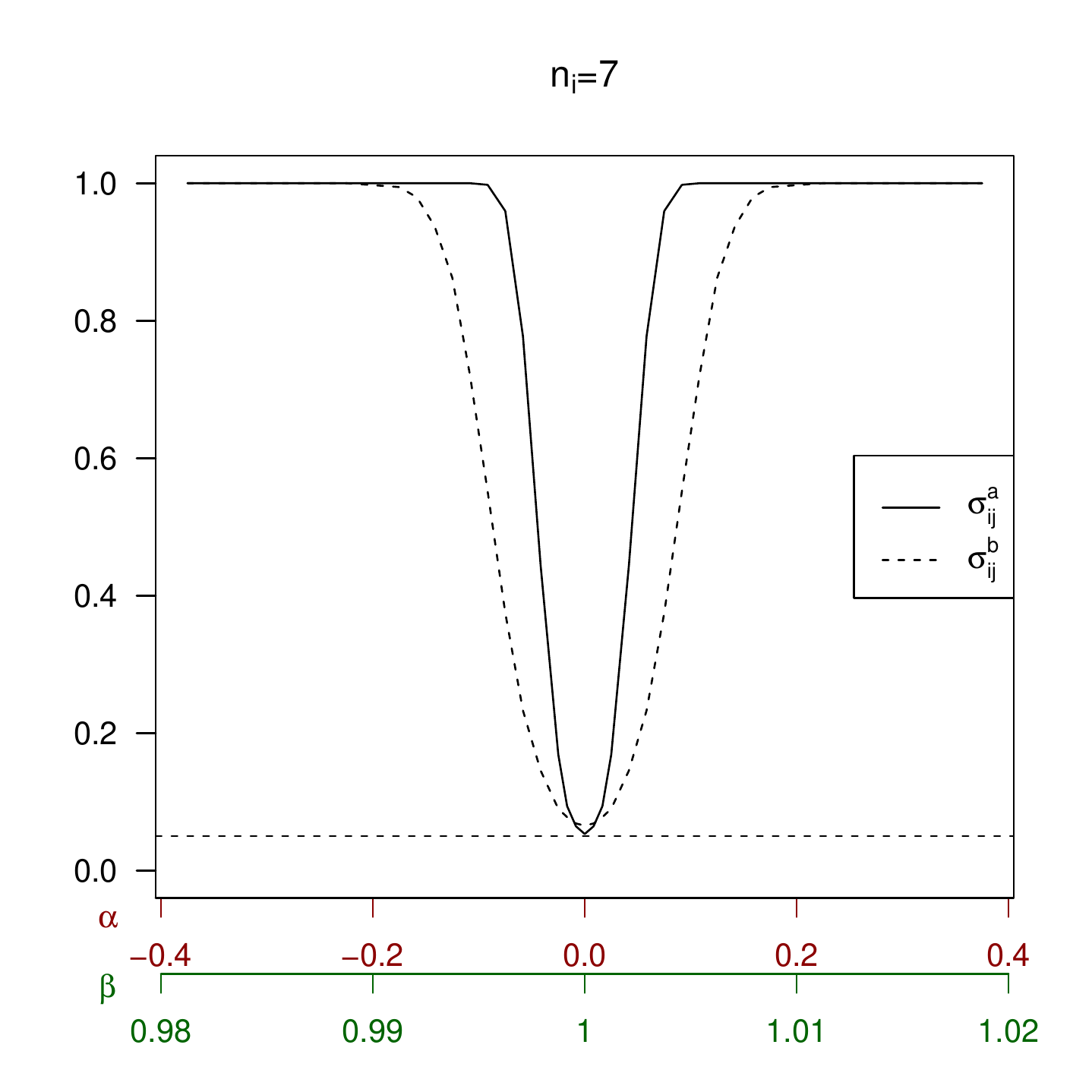} \\
			\includegraphics[width=.45\textwidth]{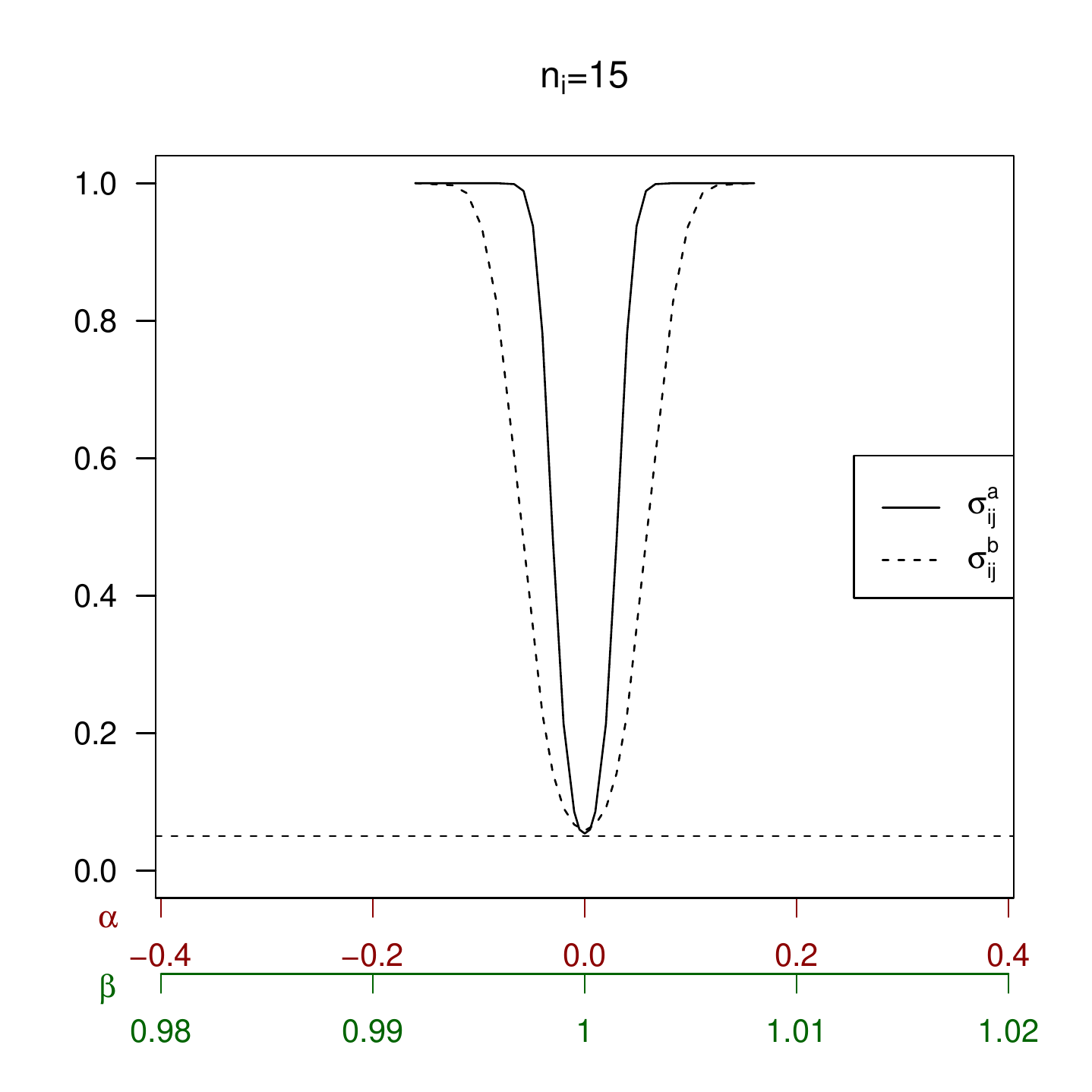} & \includegraphics[width=.45\textwidth]{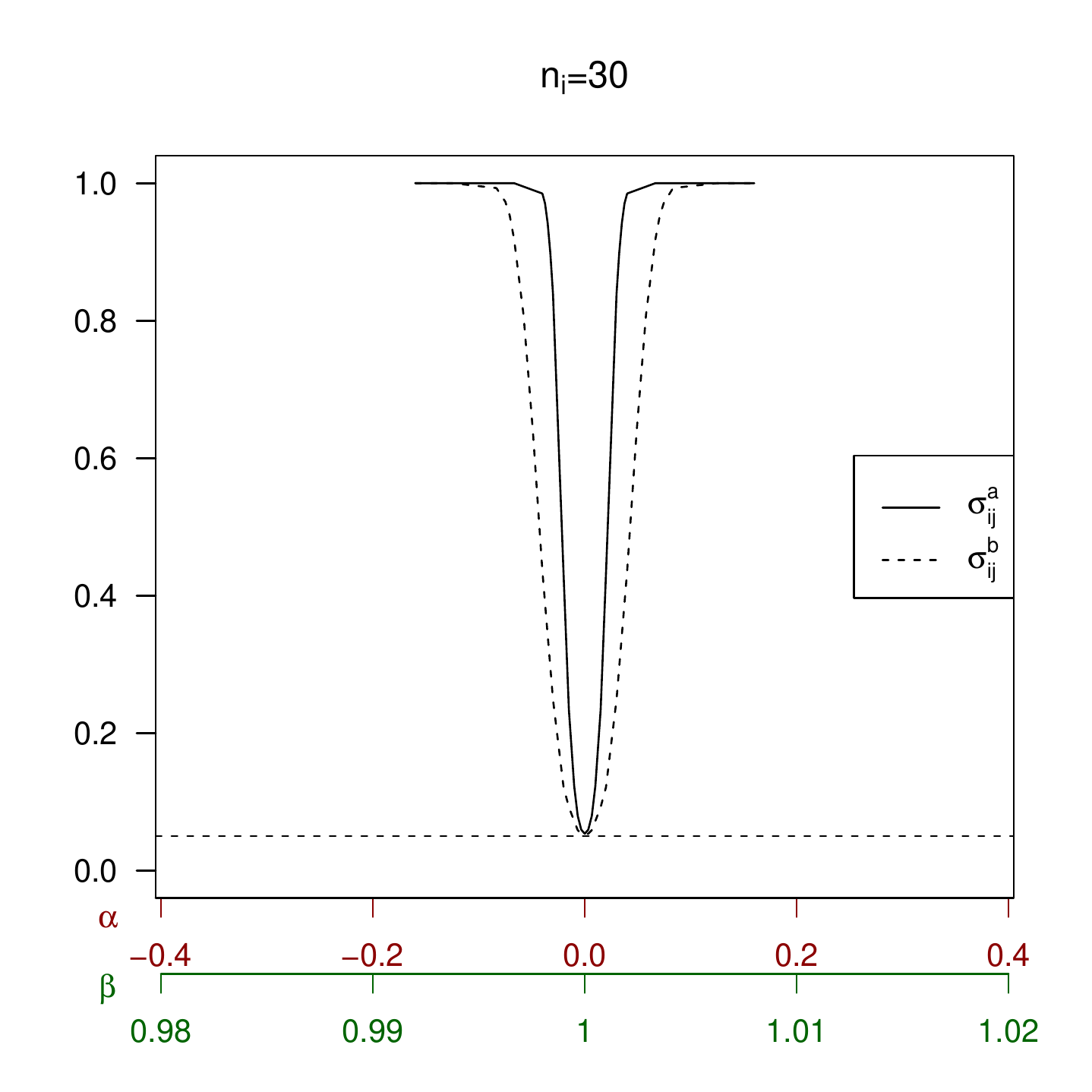}
		\end{tabular}
			\label{fig:cookPerturba.1Ap.M2058040320202562618}
	\end{figure}
In the next Section we apply the developed results for the real data set used in the stability study to show the usefulness of the proposed
methodology.

\section{Application}
In this application the GM power train developed one standard engine and its engine power was measured by 8 ($p$) laboratories at 9 ($m$) engine rotation values.
The measurements of each laboratory can be found in Appendix A. The natural variability ($\sigma^2_{x_j}$) associated with the true unobserved values was evaluated during the stability
study and the variance ($\sigma^2_{ij}$) of the measurement error corresponding to the $ith$ laboratory at the $jth$ rotation value, $i=1,\cdots,p;$ $j=1\cdots,m$, was determined by the combined variance calculated and provided by the $ith$ laboratory following the protocol proposed by ISO GUM (1995). Theses values, can also be found in Appendix A.

First, considering the EM algorithm presented in Section 2 the maximum likelihood estimates of the parameters were obtained in Table \ref{EstimativasMaximaVeross28011200}.

\begin{figure}[H]
		\centering
		\captionsetup{justification=centering}
			\caption{Simulated power for the Wald test statistics with $ \sigma_{ij}^a $ and $ \sigma_{ij}^b $ for the test\\
			$ H_0: \alpha_2=0, \beta_2=1. $}
		\begin{tabular}{@{}cc@{}}
		\label{fig:cookPerturba.1Ap.Multi.soUltimo.36142.only106.ld.52221elipse125620303032020618}
			\includegraphics[width=.5\textwidth]{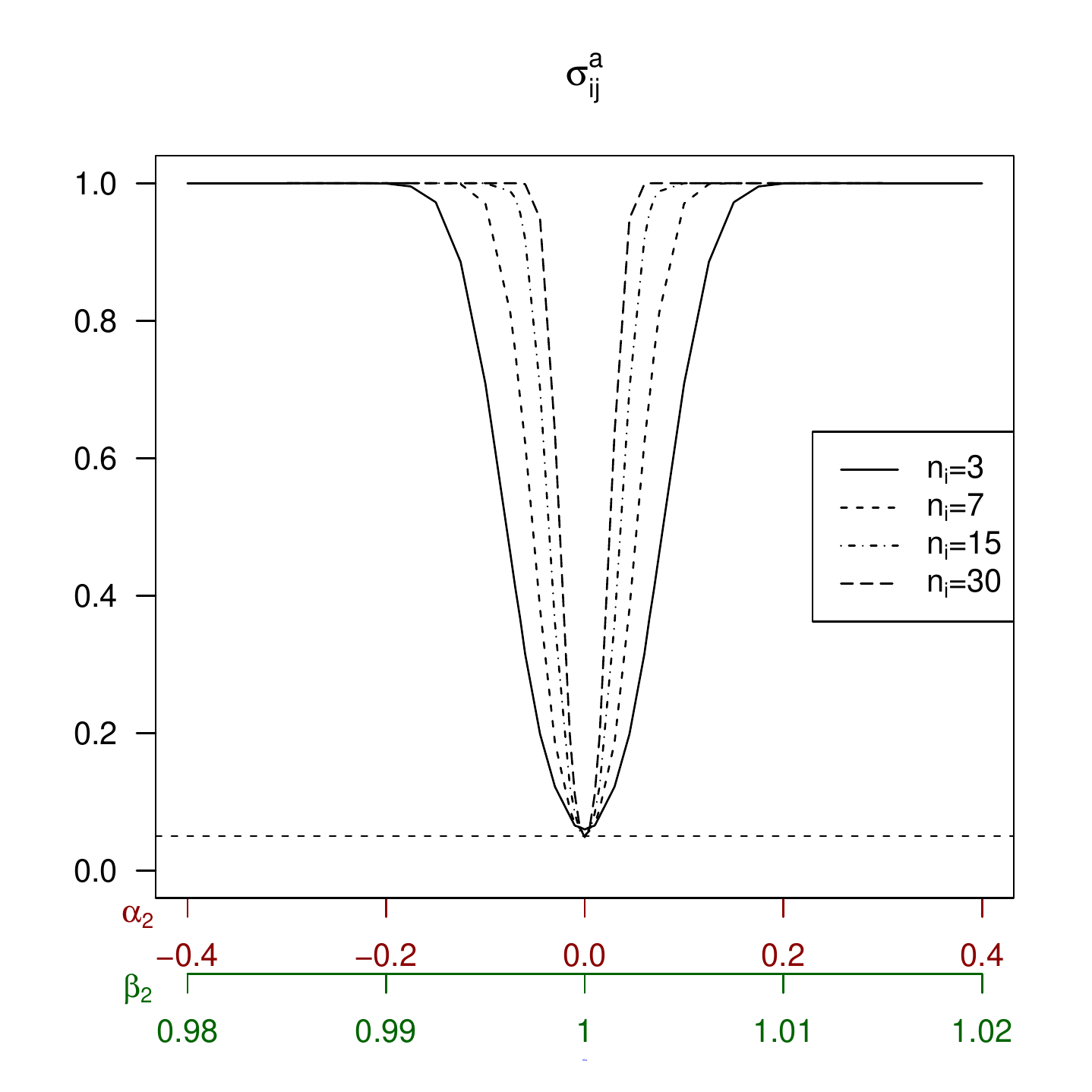} &\includegraphics[width=.5\textwidth]{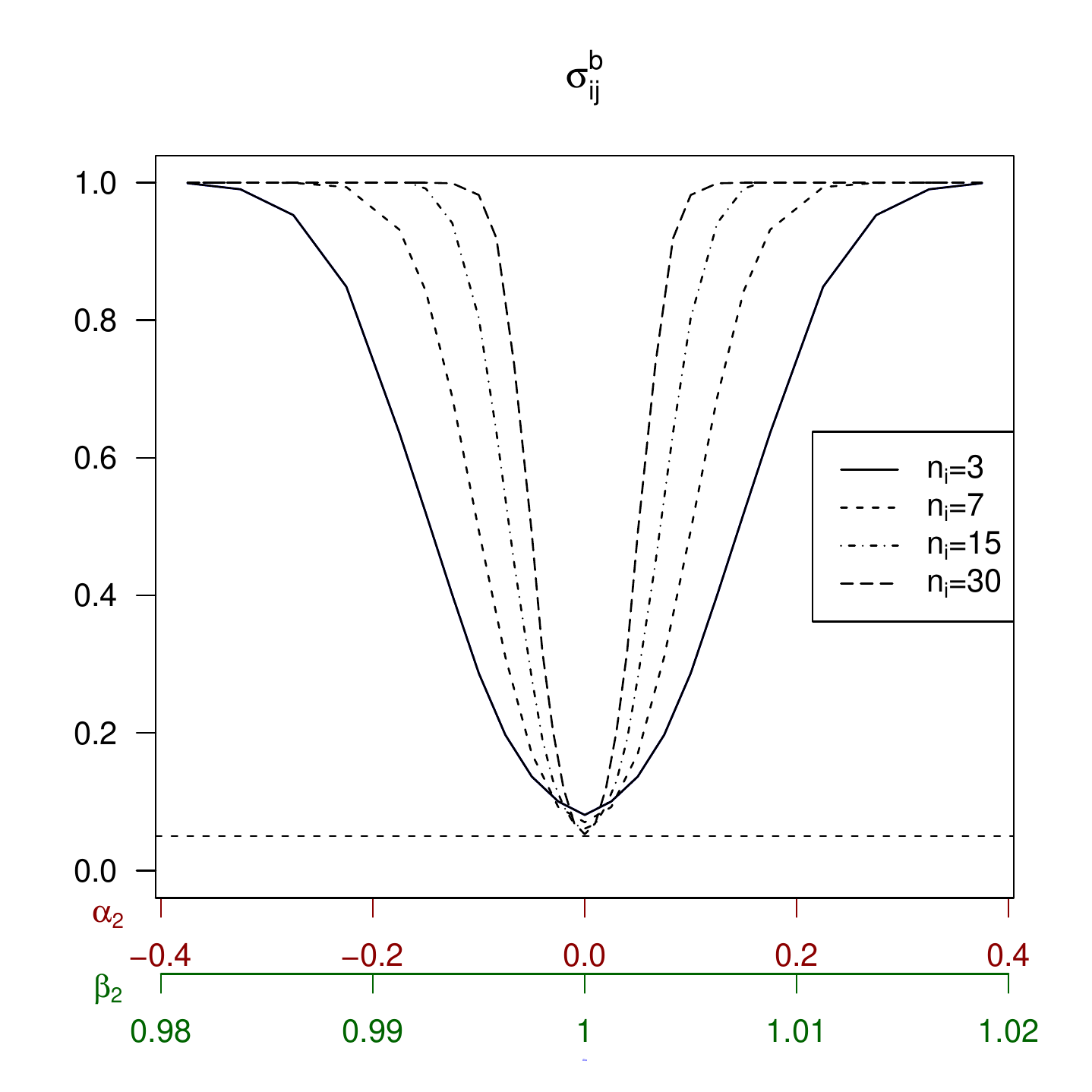}	
 		\end{tabular}
\end{figure}

\begin{table}[H]\centering
			\captionsetup{justification=centering}
		\caption{Maximum likelihood estimates of the parameters.}
		\label{EstimativasMaximaVeross28011200}
	\begin{tabular}{ccccccccc}
	\hline
	& \multicolumn{7}{c}{laboratories} \\
	\cmidrule(l){2-8}
i	&$2$&3&4&5&6&7&8\\
	 $\hat{\alpha}_i$& 0.0700 & 0.1000& 0.0658 & 0.2183& 0.1288& -0.0315 & 0.0063 \\
	$\hat{\beta}_i$&  0.9661&  0.9856 &  0.9957&  0.9871 &  0.9983&  0.9745&  0.9913\\\hline
				\end{tabular}
	\end{table}

Subsequently, we constructed the confidence regions for the seven laboratories with the confident coefficient of $99\%$ and Bonferroni  corrections, so that
the familywise error of the test is less than $1\%$. These regions can be found in Figure \ref{fig:cookPerturba_1Ap_Multi_soUltimo_36142_only106_ld_52221elipse12562618aGgplot1}. 
We can conclude visually that laboratories 4, 5 and 6 are compliant with the reference laboratory.
Moreover, all of the 7 laboratories do not have additive bias.

\begin{figure}[H]
		\centering
				\begin{tabular}{@{}cc@{}}
			\includegraphics[width=.4\textwidth]{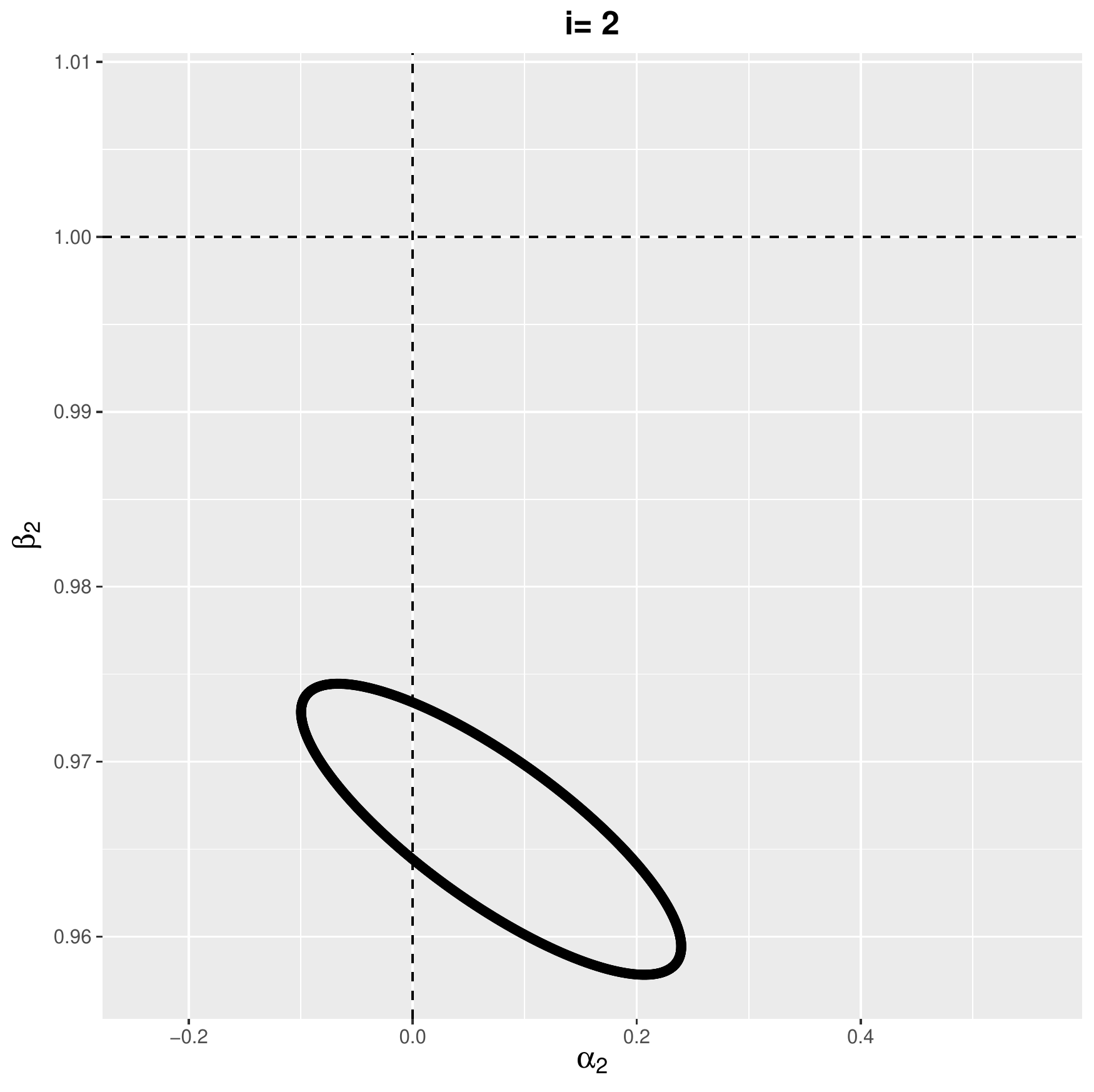}	&
			\includegraphics[width=.4\textwidth]{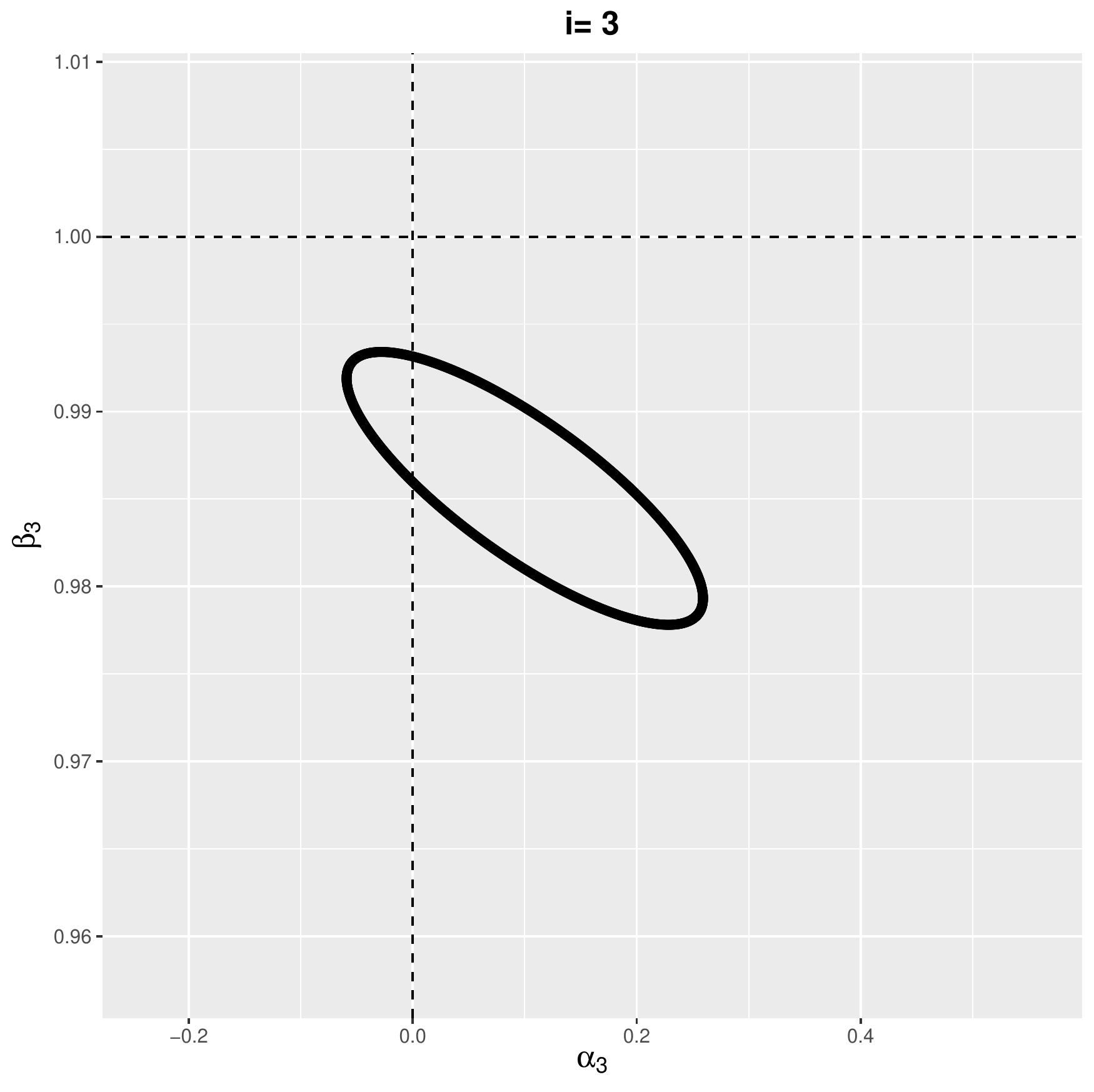}	\\
			\includegraphics[width=.4\textwidth]{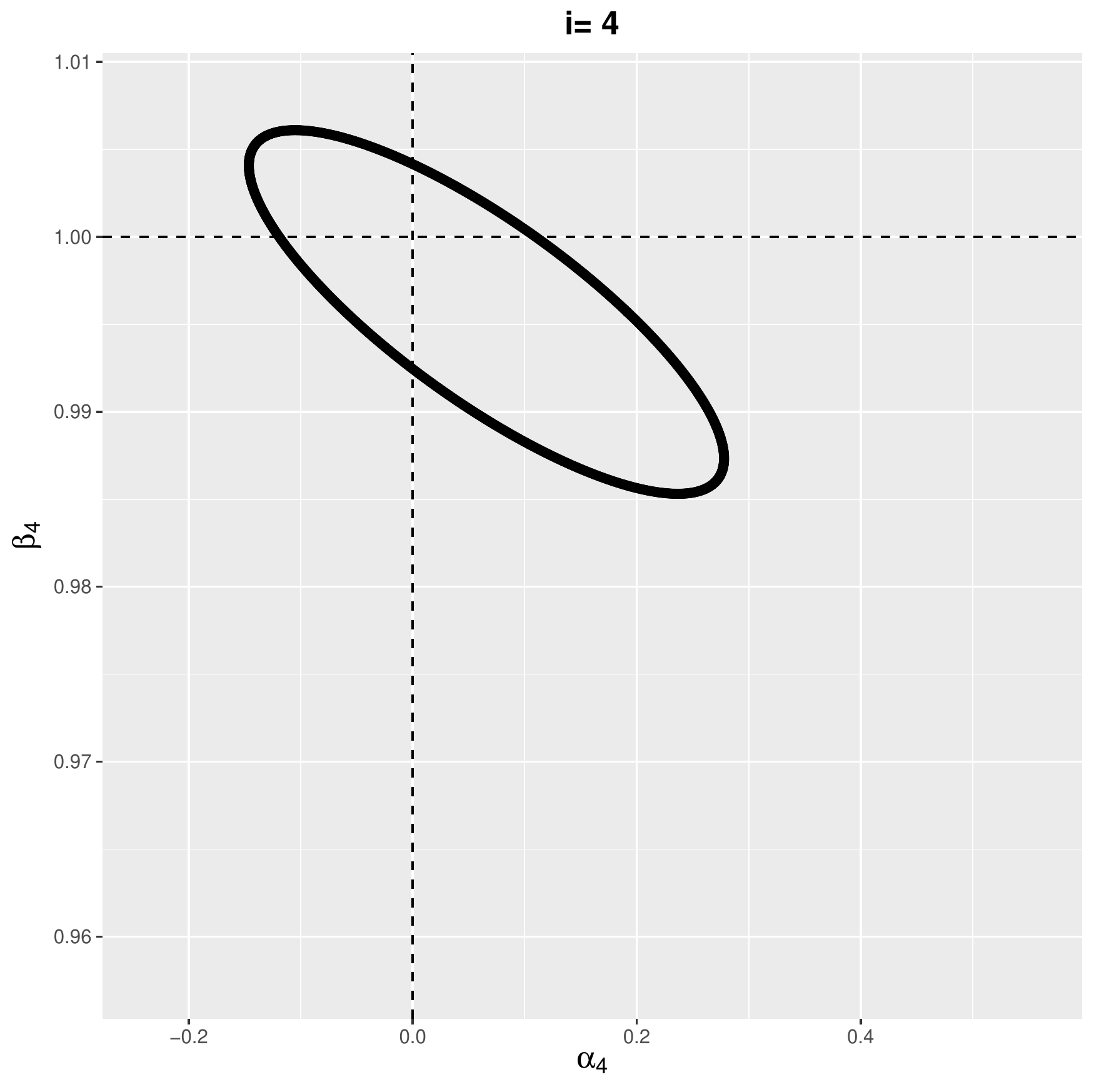}	&
			\includegraphics[width=.4\textwidth]{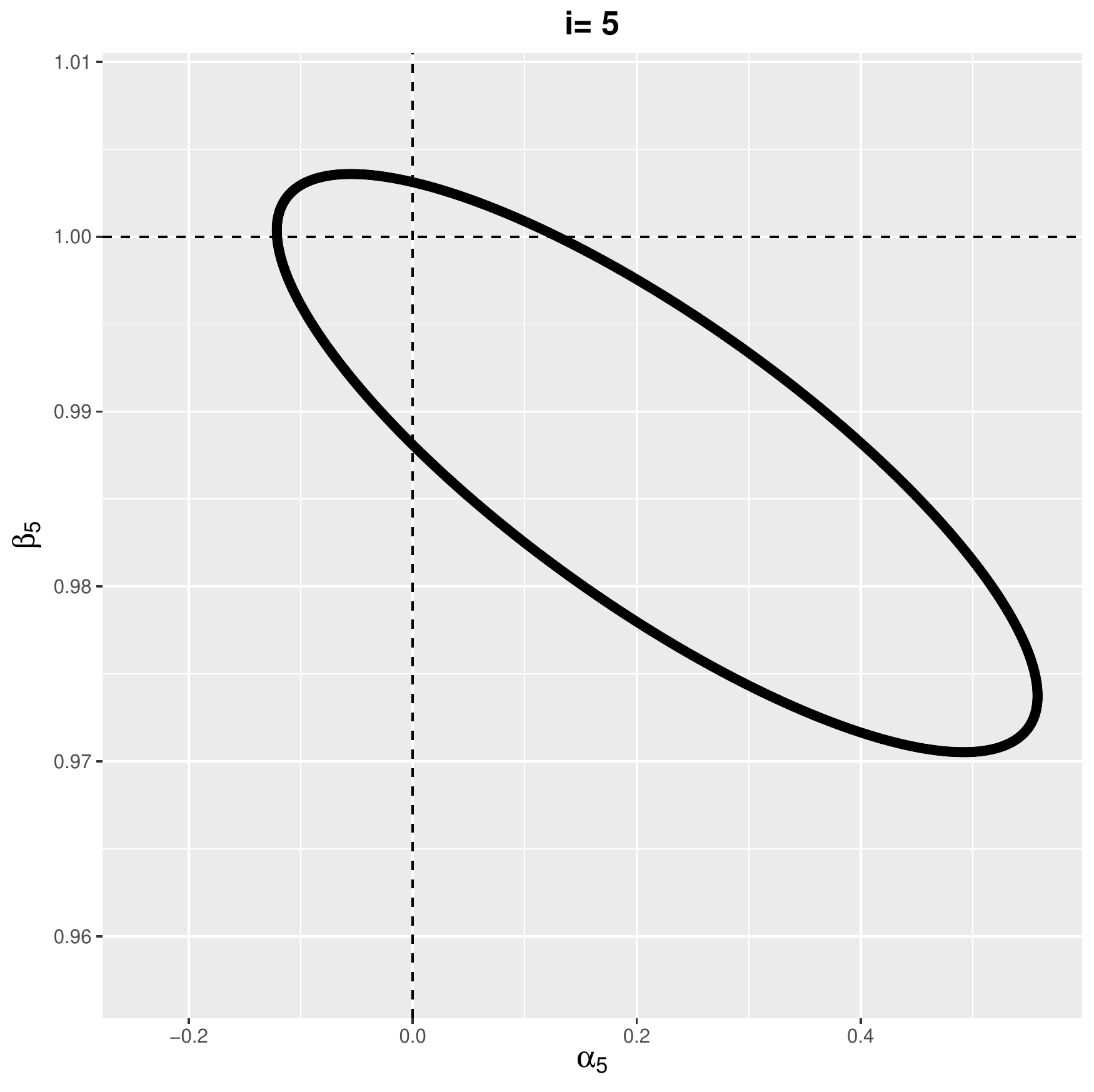}	\\	
				\includegraphics[width=.4\textwidth]{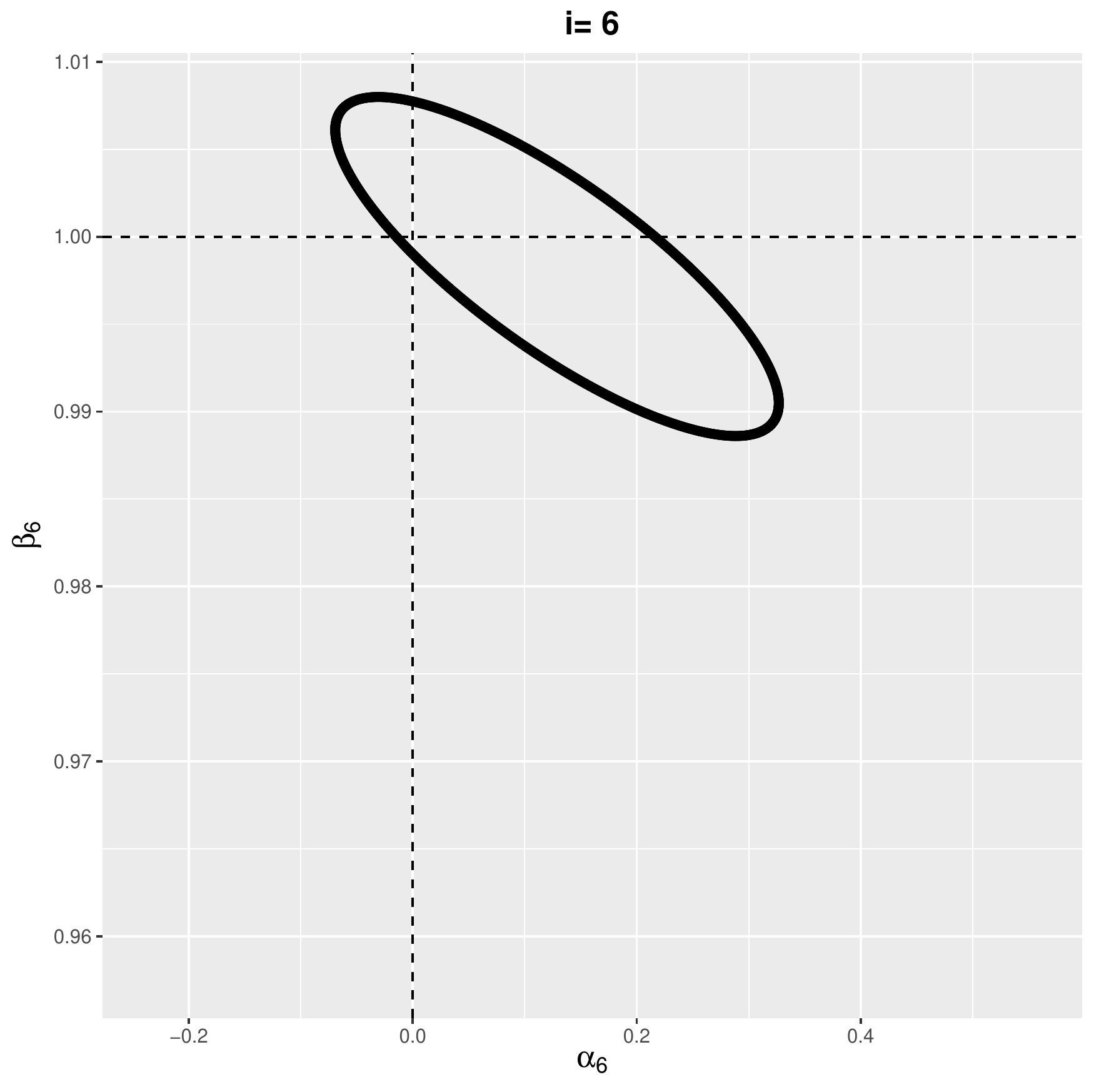}	&
			\includegraphics[width=.4\textwidth]{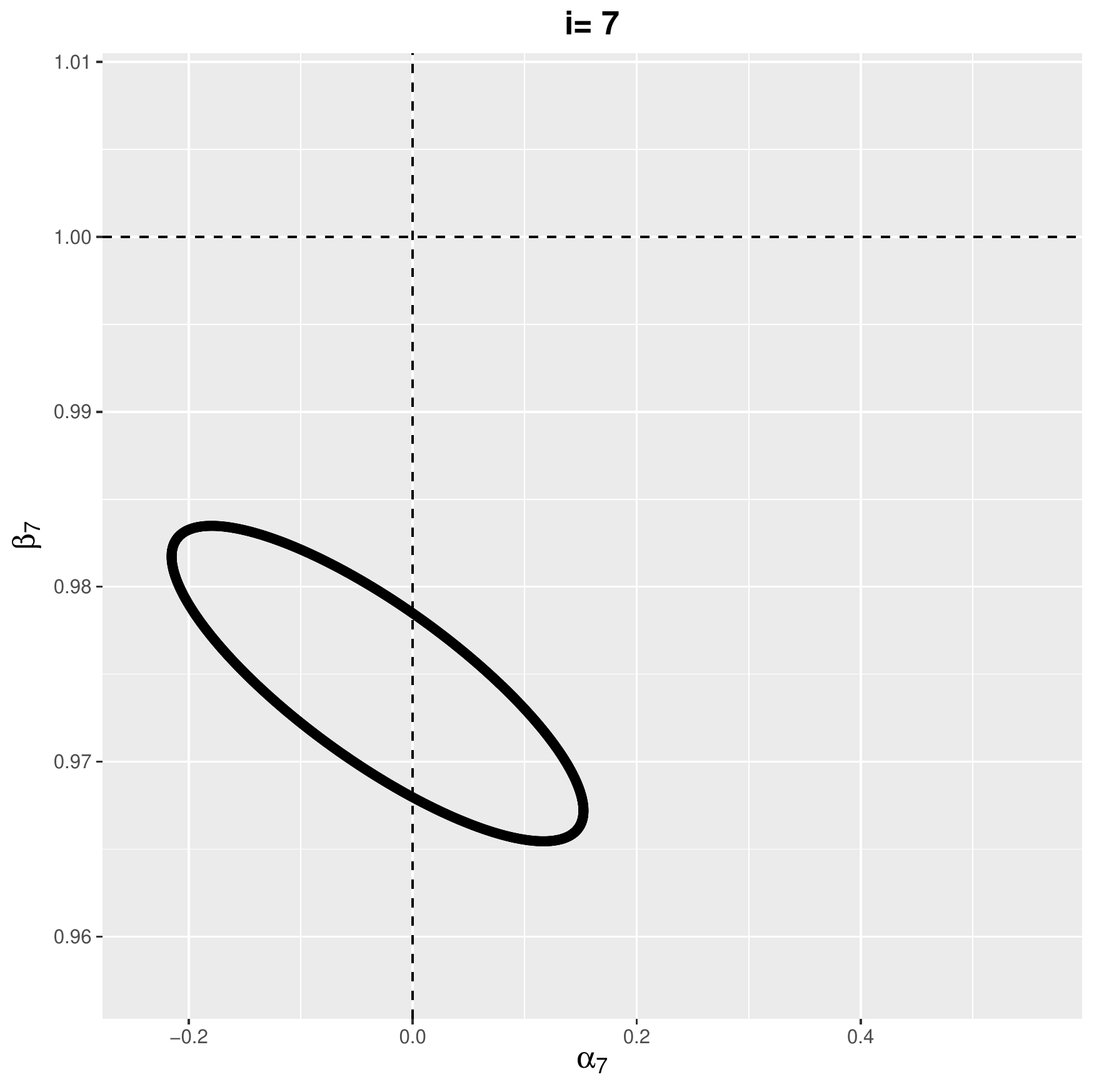}	\\		
			\includegraphics[width=.4\textwidth]{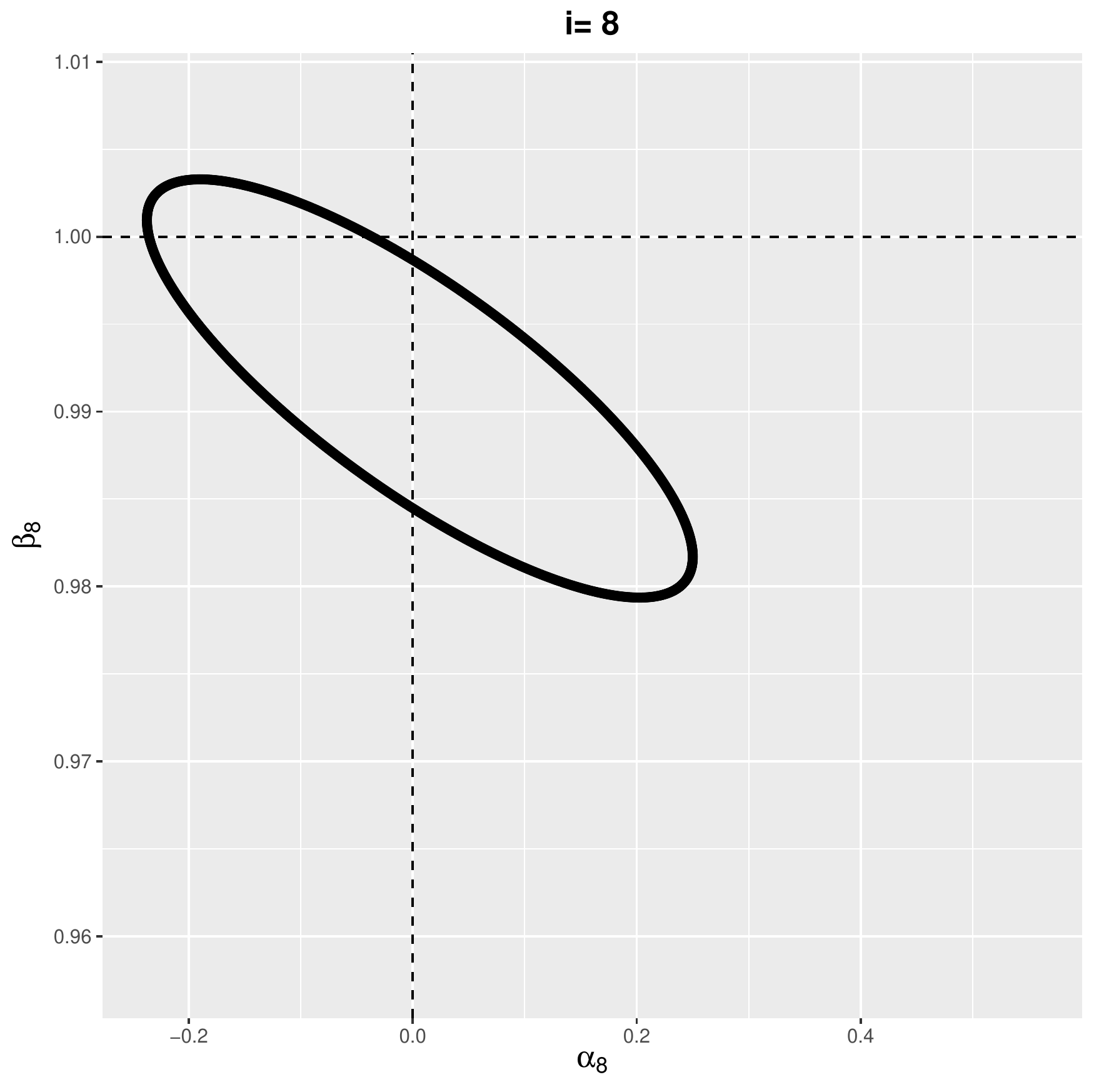}
		\end{tabular}
		\captionsetup{justification=centering}
		\caption{Joint confidence regions for the participant laboratories. } 	
		\label{fig:cookPerturba_1Ap_Multi_soUltimo_36142_only106_ld_52221elipse12562618aGgplot1}
	\end{figure}

\begin{table}[H]\centering
		\captionsetup{justification=centering}
		\caption{Wald test statistics, $Q_{w_i}$, for the hypothesis: $H_0:\alpha_i=0$, $\beta_i=1; i=2,\cdots,8$;\\ with respective $p-values$.}
		\label{principalIndividual1}
		\footnotesize
		\begin{tabular}{lcc ccc c c}\hline

	& \multicolumn{7}{c}{laboratories} \\
	\cmidrule(l){2-8}
i	&$2$&3&4&5&6&7&8\\
			$Q_{Wi}$	&	517.267900 & 69.357334 &  1.968156 &  6.639442 & 10.940891 & 324.554420 & 17.563404
			\\
			$ p-value $ & 0.000000 & 0.000000 & 0.373784 & 0.036163 & 0.004209 & 0.000000 & 0.000153 \\
			$ p-value ~(Holm)$ & 0.000000 & 0.000000 & 0.373784 & 0.072326 & 0.012628 & 0.000000 & 0.000614 \\
			$ p-value  ~(Hochberg)$ & 0.000000 & 0.000000 & 0.373784 & 0.072326 & 0.012628 & 0.000000 & 0.000614 \\
			$ p-value~ (Hommel)$ & 0.000000 & 0.000000 & 0.373784 & 0.072326 & 0.012628 & 0.000000 & 0.000614 \\
			\hline
		\end{tabular}
	\end{table}
	\normalsize

\section{Discussion}

In this work, we propose a strategy to evaluate proficiency testing results with multivariate response.  It is important to note that this is the most common case of proficiency testing. In general, the item under test is measured at different levels of values. Despite this, almost all statistical techniques used to analyze proficiency testing results consider only the univariate case. In addition, most of them do not use type B variation sources proposed by \cite{guide1995expression}.

To analyze the results of the proficiency test, we propose an ultrastructural measurement error model where the variance components are evaluated by a procedure described in \cite{guide1995expression}. As we have only one item under test, the true value is the same for all participating laboratories. This fact makes it impossible to apply the usual multivariate comparative calibration model (see, Gim\'enez and Patat (2014)), in which we have different items under test.

In the proposed ultrastructural model, there   is a natural dependency among all measurements at the same level of the item under test as described in Section 2. As a consequence, the observed Information matrix converges in probability to a random matrix. Another consequence of this dependency is that it is not possible to estimate the mean of the true value consistently, since the components of the asymptotic Information matrix with respect to the true mean value  are null. Then, to derive our strategy for comparing the results of participating laboratories, we had to develop a suitable asymptotic theory based on the smoothness of the likelihood function, as developed by\cite{weiss1971asymptotic}, \cite{weiss1973asymptotic} and \cite{sweeting1980uniform}.

Due to the fact that the asymptotic information matrix $(W)$ is random, it is not possible to apply Slutsky's theorem to check the convergence in distribution of the transformation of the score function sequence and the observed information matrix sequence, which is necessary to develop the usual asymptotic theory. To address this problem, we derived the asymptotic joint distribution of the score function and the observed Fisher information matrix (see, Theorem \ref{Score_limit}). Next, we used the smoothness of the likelihood function and the continuous transformation theorem to arrive at the asymptotic distribution of Wald's statistic. A curious point is the fact that the maximum likelihood estimator has no asymptotic normal distribution (see, Equation \ref{fan_MLE}).

To assess the behavior of asymptotic results, a  simulation study was developed. In general, we conclude that the performance of the asymptotic results are closely related to the sample size and the magnitude of the variance components. Even with small sample size, the empirical and nominal values of the  significance levels are close (see, Tables \ref{yearlyprivatizatio01311212030320202n} and  \ref{yearlyprivatizatio0131na}). Moreover, the empirical power funcion has the same behavior (see, Figures \ref{poderTodos} 
and
\ref{fig:cookPerturba.1Ap.Multi.soUltimo.36142.only106.ld.52221elipse125620303032020618} and Figure 1 in Online Resource- Section 3). As the variance components of the reference laboratory are known before the start of the PT program, we can use the empirical power function to estimate the sample size.

To illustrate the developed methodology, we analyzed the results of the proficiency test for the power measurements of an engine in the Application Section. In the real data set considered here, we have 8 laboratories including the reference laboratory. At the beginning of the program, the reference laboratory evaluated the stability of the engine under test and determined the component of variance related to the true value. To ensure comparability of results, the reference laboratory measured the engine at the beginning and end of the PT program. Each participating laboratory reported its measurements and respective uncertainties (Type A and Type B). In the sequel, the statistical coordinator of the PT program compared the results of the participant laboratories with the reference laboratory.

The results of the participant laboratories were compared with the results of the reference laboratory using Wald statistics, as presented in section 4. Initially, we considered the test given in (\ref{eqWald11}) to assess equivalence among laboratories measurements. As we rejected the hypothesis of equivalence among laboratories measurement, we compared the results of each participant laboratory with the reference value. For this, we proposed a joint confidence region for the bias parameters related to respective participant laboratory (see, Figure \ref{fig:cookPerturba_1Ap_Multi_soUltimo_36142_only106_ld_52221elipse12562618a}). Based on the joint confidence region, it was possible to assess the consistency of the results of the participant laboratory with respect to the reference laboratory. Moreover, if the participant laboratory results are not consistent, we can identify which bias parameters are significant.

In summary, the measurement results comparison strategy proposed in this work can be applied in any situation where participant laboratories measure the same item and we have a reference laboratory to compare the results.

\section{acknowledgements}
The research was financed in part by the Coordenação de Aperfeiçoamento de Pessoal de Nível Superior - Brasil (CAPES) - Finance Code 001.
\vspace{5mm}
\medskip
\medskip

\noindent
{\Large{\bf Appendix A - Engine Power Data Set}}
\medskip
\medskip
\medskip

In this Section we present the data set used to illustrate the developed methodology which
consists of the measurements of the power of the engine in $9$
points of rotation by $8$ laboratories.
 We are not going
to identify the
laboratories in the data set,
as it is confidential.
\medskip
\medskip

	\begin{center}
		\footnotesize
		\begin{longtable}{ccc ccc ccc ccc cc}
			\captionsetup{justification=centering}  
			\caption[]{Engine Power Data Set.}
			\endfirsthead
			\multicolumn{14}{c}%
			{{ \tablename\ \thetable{}: Engine Power Data Set.}} \\
			\midrule[1pt]
			\midrule[1pt]
			\endhead
			\midrule[1pt]
			\endfoot
			\toprule[1.5pt]
			\endlastfoot
			& $Y_{1j1}$ & $Y_{1j2}$ & $Y_{1j3}$ & $Y_{1j4}$ & $Y_{1j5}$ & $Y_{2j1}$ & $Y_{2j2}$ & $Y_{2j3}$ & $Y_{2j4}$ & $Y_{2j5}$ & $Y_{2j6}$ & $Y_{2j7}$ & $Y_{2j8}$ \\
			\hline
			1200 & 8.89 & 8.83 & 8.86 & 8.85 & 8.88 & 8.57 & 8.56 & 8.58 & 8.64 & 8.64 & 8.60 & 8.59 & 8.56 \\
			2000 & 15.84 & 15.80 & 15.80 & 15.79 & 15.81 & 15.28 & 15.27 & 15.35 & 15.36 & 15.38 & 15.29 & 15.24 & 15.27 \\
			3000 & 26.84 & 26.61 & 26.86 & 26.85 & 26.92 & 26.12 & 26.12 & 26.09 & 26.12 & 26.22 & 26.12 & 25.95 & 26.14 \\
			3600 & 31.41 & 31.31 & 31.40 & 31.40 & 31.50 & 30.48 & 30.45 & 30.41 & 30.52 & 30.48 & 30.41 & 30.25 & 30.32 \\
			4400 & 37.19 & 37.12 & 37.24 & 37.17 & 37.32 & 36.46 & 36.33 & 36.25 & 36.36 & 36.45 & 36.41 & 36.21 & 36.20 \\
			5200 & 44.35 & 44.35 & 44.28 & 44.31 & 44.40 & 43.29 & 43.22 & 43.11 & 43.21 & 43.34 & 43.28 & 43.05 & 43.05 \\
			5600 & 47.49 & 47.33 & 47.56 & 47.60 & 47.79 & 46.29 & 46.21 & 46.11 & 46.17 & 46.18 & 45.94 & 46.01 & 45.97 \\
			6000 & 49.92 & 49.76 & 49.98 & 49.90 & 49.96 & 48.01 & 47.92 & 47.90 & 48.00 & 47.88 & 47.80 & 47.68 & 47.71 \\
			6400 & 50.92 & 50.74 & 50.89 & 50.84 & 50.94 & 48.85 & 48.84 & 48.75 & 48.86 & 48.85 & 48.66 & 48.57 & 48.45 \\
			\hline
			& $Y_{2j9}$ & $Y_{2j10}$ & $Y_{2j11}$ & $Y_{2j12}$ & $Y_{2j13}$ & $Y_{2j14}$ & $Y_{2j15}$ & $Y_{2j16}$ & $Y_{2j17}$ & $Y_{2j18}$ & $Y_{2j19}$ & $Y_{2j20}$ & $Y_{2j21}$ \\
			\hline
			1200 & 8.57 & 8.64 & 8.59 & 8.60 & 8.67 & 8.72 & 8.79 & 8.79 & 8.79 & 8.73 & 8.60 & 8.60 & 8.46 \\
			2000 & 15.31 & 15.34 & 15.32 & 15.37 & 15.35 & 15.57 & 15.60 & 15.60 & 15.54 & 15.53 & 15.26 & 15.26 & 15.10 \\
			3000 & 26.06 & 25.93 & 26.03 & 26.08 & 26.13 & 26.34 & 26.27 & 26.25 & 26.38 & 26.49 & 26.05 & 26.19 & 25.76 \\
			3600 & 30.42 & 30.52 & 30.45 & 30.44 & 30.65 & 30.91 & 30.93 & 30.97 & 30.94 & 30.86 & 30.26 & 30.57 & 29.98 \\
			4400 & 36.16 & 36.33 & 36.17 & 36.17 & 36.34 & 36.58 & 36.44 & 36.52 & 36.82 & 36.66 & 36.23 & 36.32 & 35.80 \\
			5200 & 42.98 & 43.11 & 43.06 & 43.06 & 43.21 & 43.45 & 43.37 & 43.50 & 43.65 & 43.65 & 43.23 & 43.20 & 42.54 \\
			5600 & 46.08 & 46.03 & 46.19 & 46.29 & 46.58 & 46.64 & 46.69 & 46.71 & 46.73 & 46.03 & 46.03 & 45.34 & 45.36 \\
			6000 & 47.64 & 47.90 & 47.79 & 47.85 & 47.99 & 48.13 & 48.35 & 48.41 & 48.51 & 48.41 & 47.60 & 47.67 & 46.93 \\
			6400 & 48.58 & 48.66 & 48.86 & 48.81 & 49.00 & 49.09 & 49.29 & 49.33 & 49.42 & 49.33 & 48.44 & 48.38 & 47.79 \\
			\hline
			& $Y_{2j22}$ & $Y_{2j23}$ & $Y_{3j1}$ & $Y_{3j2}$ & $Y_{3j3}$ & $Y_{3j4}$ & $Y_{3j5}$ & $Y_{3j6}$ & $Y_{3j7}$ & $Y_{3j8}$ & $Y_{3j9}$ & $Y_{3j10}$ & $Y_{3j11}$ \\
			\hline
			1200 & 8.47 & 8.52 & 8.85 & 8.76 & 8.77 & 8.80 & 8.77 & 8.77 & 8.78 & 8.78 & 8.76 & 8.77 & 8.77 \\
			2000 & 15.05 & 15.18 & 15.85 & 15.78 & 15.81 & 15.71 & 15.76 & 15.75 & 15.76 & 15.75 & 15.73 & 15.76 & 15.73 \\
			3000 & 25.61 & 25.78 & 26.68 & 26.68 & 26.68 & 26.52 & 26.59 & 26.61 & 26.52 & 26.50 & 26.61 & 26.56 & 26.73 \\
			3600 & 29.90 & 30.13 & 31.33 & 31.22 & 31.23 & 31.21 & 31.27 & 31.25 & 31.28 & 31.26 & 31.31 & 31.26 & 31.25 \\
			4400 & 35.67 & 35.73 & 37.02 & 36.88 & 36.93 & 36.89 & 36.89 & 36.88 & 36.92 & 36.94 & 36.93 & 36.91 & 36.91 \\
			5200 & 42.56 & 42.56 & 43.90 & 43.81 & 43.88 & 43.75 & 43.79 & 43.74 & 43.80 & 43.86 & 43.82 & 43.84 & 43.82 \\
			5600 & 45.55 & 45.70 & 47.00 & 46.99 & 47.01 & 46.93 & 46.99 & 46.94 & 46.99 & 47.06 & 47.05 & 47.03 & 47.02 \\
			6000 & 46.84 & 47.20 & 48.78 & 48.81 & 48.87 & 48.86 & 48.83 & 48.79 & 48.81 & 48.90 & 48.86 & 48.84 & 48.82 \\
			6400 & 47.71 & 48.01 & 49.86 & 49.87 & 49.91 & 49.88 & 49.89 & 49.85 & 49.90 & 49.92 & 49.90 & 49.88 & 49.85 \\
			\hline
			& $Y_{3j12}$ & $Y_{3j13}$ & $Y_{3j14}$ & $Y_{3j15}$ & $Y_{3j16}$ & $Y_{3j17}$ & $Y_{3j18}$ & $Y_{4j1}$ & $Y_{4j2}$ & $Y_{4j3}$ & $Y_{4j4}$ & $Y_{4j5}$ & $Y_{4j6}$ \\
			\hline
			1200 & 8.84 & 8.79 & 8.79 & 8.80 & 8.81 & 8.89 & 8.81 & 8.73 & 8.84 & 8.85 & 8.77 & 8.78 & 8.88 \\
			2000 & 15.81 & 15.78 & 15.81 & 15.82 & 15.74 & 15.82 & 15.82 & 15.68 & 15.88 & 15.89 & 15.82 & 16.00 & 16.00 \\
			3000 & 26.81 & 26.72 & 26.76 & 26.73 & 26.77 & 26.82 & 26.82 & 26.60 & 26.91 & 26.82 & 26.84 & 27.04 & 27.06 \\
			3600 & 31.29 & 31.28 & 31.33 & 31.15 & 31.26 & 31.26 & 31.28 & 31.65 & 31.49 & 31.49 & 31.31 & 31.51 & 31.60 \\
			4400 & 36.90 & 36.86 & 36.87 & 36.88 & 36.95 & 36.98 & 36.96 & 37.51 & 37.33 & 37.32 & 37.13 & 37.31 & 37.35 \\
			5200 & 43.83 & 43.80 & 43.78 & 43.78 & 43.88 & 43.87 & 43.80 & 44.37 & 44.29 & 44.29 & 44.05 & 44.23 & 44.25 \\
			5600 & 47.06 & 46.98 & 46.96 & 46.99 & 47.06 & 47.09 & 47.00 & 47.44 & 47.43 & 47.47 & 47.25 & 47.30 & 47.32 \\
			6000 & 48.88 & 48.78 & 48.76 & 48.78 & 48.86 & 48.88 & 48.74 & 49.30 & 49.31 & 49.36 & 49.13 & 49.19 & 49.14 \\
			6400 & 49.84 & 49.76 & 49.73 & 49.75 & 49.83 & 49.85 & 49.73 & 50.20 & 50.24 & 50.23 & 50.05 & 50.09 & 50.08 \\
			\hline
			& $Y_{4j7}$ & $Y_{4j8}$ & $Y_{4j9}$ & $Y_{5j1}$ & $Y_{5j2}$ & $Y_{5j3}$ & $Y_{5j4}$ & $Y_{5j5}$ & $Y_{5j6}$ & $Y_{5j7}$ & $Y_{5j8}$ & $Y_{5j9}$ & $Y_{5j10}$ \\
			\hline
			1200 & 8.90 & 8.92 & 8.93 & 8.70 & 8.68 & 9.13 & 9.07 & 9.04 & 9.03 & 8.82 & 8.75 & 8.92 & 8.95 \\
			2000 & 15.94 & 16.00 & 16.05 & 15.46 & 15.50 & 16.24 & 16.35 & 16.14 & 16.22 & 15.58 & 15.58 & 16.05 & 15.96 \\
			3000 & 26.85 & 26.93 & 26.96 & 25.94 & 25.91 & 27.24 & 27.41 & 27.38 & 27.41 & 26.16 & 25.96 & 27.18 & 27.05 \\
			3600 & 31.56 & 31.56 & 31.60 & 30.71 & 30.85 & 32.15 & 32.27 & 32.26 & 32.26 & 30.99 & 30.83 & 31.95 & 31.87 \\
			4400 & 37.32 & 37.36 & 37.45 & 35.93 & 35.92 & 38.02 & 38.10 & 37.99 & 38.16 & 36.54 & 36.26 & 37.59 & 37.57 \\
			5200 & 44.34 & 44.40 & 44.39 & 42.47 & 42.56 & 45.16 & 45.27 & 44.98 & 45.28 & 43.08 & 42.92 & 44.67 & 44.70 \\
			5600 & 47.54 & 47.52 & 47.57 & 45.65 & 45.70 & 48.38 & 48.40 & 48.13 & 48.40 & 46.12 & 45.94 & 47.82 & 47.96 \\
			6000 & 49.41 & 49.44 & 49.43 & 47.30 & 47.42 & 49.95 & 50.08 & 49.89 & 50.36 & 47.71 & 47.63 & 49.51 & 49.40 \\
			6400 & 50.33 & 50.35 & 50.34 & 48.11 & 48.19 & 50.76 & 50.24 & 50.49 & 50.51 & 48.51 & 48.33 & 49.99 & 50.38 \\
			\hline
			& $Y_{5j11}$ & $Y_{5j12}$ & $Y_{6j1}$ & $Y_{6j2}$ & $Y_{6j3}$ & $Y_{6j4}$ & $Y_{6j5}$ & $Y_{6j6}$ & $Y_{6j7}$ & $Y_{6j8}$ & $Y_{6j9}$ & $Y_{6j10}$ & $Y_{6j11}$ \\
			\hline
			1200 & 8.85 & 8.87 & 9.00 & 8.90 & 8.90 & 9.00 & 9.00 & 8.90 & 8.90 & 8.90 & 9.00 & 8.90 & 8.90 \\
			2000 & 15.94 & 15.95 & 16.00 & 15.90 & 16.10 & 16.10 & 16.10 & 16.00 & 16.00 & 15.90 & 16.10 & 16.00 & 16.00 \\
			3000 & 26.81 & 26.84 & 27.10 & 27.00 & 27.10 & 27.10 & 27.10 & 27.10 & 27.00 & 27.00 & 27.30 & 27.10 & 27.10 \\
			3600 & 31.73 & 31.67 & 31.70 & 31.60 & 31.70 & 31.70 & 31.80 & 31.80 & 31.50 & 31.60 & 31.80 & 31.70 & 31.70 \\
			4400 & 37.41 & 37.58 & 37.60 & 37.50 & 37.70 & 37.70 & 37.70 & 37.60 & 37.50 & 37.60 & 37.90 & 37.70 & 37.60 \\
			5200 & 44.18 & 44.15 & 44.80 & 44.60 & 44.90 & 44.70 & 44.80 & 44.80 & 44.60 & 44.60 & 44.90 & 44.90 & 44.70 \\
			5600 & 47.39 & 47.41 & 47.60 & 47.40 & 47.70 & 47.70 & 47.60 & 47.80 & 47.30 & 47.30 & 47.80 & 47.90 & 47.50 \\
			6000 & 49.23 & 49.33 & 49.10 & 49.10 & 49.30 & 49.40 & 49.30 & 49.30 & 49.00 & 49.00 & 49.40 & 49.50 & 49.20 \\
			6400 & 49.68 & 49.72 & 49.80 & 49.90 & 50.00 & 50.10 & 50.20 & 50.10 & 49.70 & 49.70 & 50.20 & 50.20 & 49.90 \\
			\hline
			& $Y_{6j12}$ & $Y_{6j13}$ & $Y_{6j14}$ & $Y_{6j15}$ & $Y_{6j16}$ & $Y_{7j1}$ & $Y_{7j2}$ & $Y_{7j3}$ & $Y_{7j4}$ & $Y_{7j5}$ & $Y_{7j6}$ & $Y_{7j7}$ & $Y_{7j8}$ \\
			\hline
			1200 & 8.90 & 8.90 & 8.90 & 8.90 & 8.90 & 8.70 & 8.70 & 8.50 & 8.50 & 8.50 & 8.60 & 8.60 & 8.60 \\
			2000 & 16.00 & 16.10 & 16.00 & 16.00 & 16.00 & 15.60 & 15.60 & 15.40 & 15.40 & 15.40 & 15.40 & 15.50 & 15.40 \\
			3000 & 27.10 & 27.20 & 27.10 & 27.00 & 27.10 & 26.10 & 25.80 & 25.60 & 25.60 & 25.60 & 25.70 & 26.00 & 25.80 \\
			3600 & 31.70 & 31.80 & 31.70 & 31.60 & 31.70 & 30.70 & 30.60 & 30.20 & 30.10 & 30.40 & 30.60 & 30.50 & 30.50 \\
			4400 & 37.60 & 37.70 & 37.60 & 37.70 & 37.70 & 36.50 & 36.60 & 36.30 & 36.30 & 36.40 & 36.50 & 36.60 & 36.50 \\
			5200 & 44.70 & 44.90 & 44.70 & 44.80 & 44.70 & 43.80 & 43.60 & 43.40 & 43.10 & 42.80 & 43.30 & 43.50 & 43.50 \\
			5600 & 47.60 & 47.80 & 47.80 & 47.50 & 47.50 & 46.20 & 46.10 & 46.00 & 45.70 & 46.40 & 46.60 & 46.50 & 46.60 \\
			6000 & 49.30 & 49.50 & 49.50 & 49.20 & 49.10 & 48.70 & 48.30 & 47.70 & 48.10 & 47.40 & 48.30 & 48.30 & 48.50 \\
			6400 & 50.10 & 50.20 & 50.20 & 50.00 & 49.90 & 49.70 & 49.60 & 48.60 & 49.00 & 49.10 & 49.10 & 49.40 & 49.40 \\
			\hline
			& $Y_{7j9}$ & $Y_{7j10}$ & $Y_{7j11}$ & $Y_{7j12}$ & $Y_{7j13}$ & $Y_{7j14}$ & $Y_{7j15}$ & $Y_{7j16}$ & $Y_{7j17}$ & $Y_{7j18}$ & $Y_{7j19}$ & $Y_{7j20}$ & $Y_{7j21}$ \\
			\hline
			1200 & 8.60 & 8.60 & 8.60 & 8.60 & 8.60 & 8.60 & 8.70 & 8.60 & 8.60 & 8.60 & 8.50 & 8.60 & 8.60 \\
			2000 & 15.40 & 15.50 & 15.40 & 15.50 & 15.40 & 15.50 & 15.60 & 15.50 & 15.40 & 15.50 & 15.50 & 15.40 & 15.40 \\
			3000 & 25.80 & 25.90 & 26.00 & 25.70 & 25.70 & 25.90 & 25.90 & 25.90 & 26.00 & 25.80 & 25.90 & 25.80 & 25.80 \\
			3600 & 30.60 & 30.50 & 30.40 & 30.40 & 30.40 & 30.60 & 30.60 & 30.50 & 30.50 & 30.60 & 30.60 & 30.60 & 30.60 \\
			4400 & 36.60 & 36.50 & 36.60 & 36.60 & 36.70 & 36.70 & 36.60 & 36.70 & 36.70 & 36.70 & 36.60 & 36.60 & 36.70 \\
			5200 & 43.60 & 44.00 & 44.20 & 43.50 & 43.50 & 43.70 & 43.70 & 43.70 & 43.80 & 43.70 & 43.70 & 43.70 & 43.60 \\
			5600 & 47.00 & 47.50 & 46.70 & 46.70 & 46.60 & 46.80 & 46.50 & 47.00 & 46.50 & 46.60 & 46.10 & 46.80 & 46.70 \\
			6000 & 48.30 & 49.20 & 49.10 & 48.60 & 48.40 & 48.50 & 48.60 & 48.50 & 48.70 & 48.50 & 48.70 & 48.40 & 48.50 \\
			6400 & 49.30 & 50.00 & 50.20 & 49.30 & 49.20 & 49.40 & 49.10 & 49.10 & 49.30 & 49.50 & 49.40 & 49.40 & 49.10 \\
			\hline
			& $Y_{7j22}$ & $Y_{7j23}$ & $Y_{7j24}$ & $Y_{7j25}$ & $Y_{7j26}$ & $Y_{8j1}$ & $Y_{8j2}$ & $Y_{8j3}$ & $Y_{8j4}$ & $Y_{8j5}$ & $Y_{8j6}$ & $Y_{8j7}$ & $Y_{8j8}$ \\
			\hline
			1200 & 8.70 & 8.60 & 8.60 & 8.60 & 8.60 & 8.70 & 8.70 & 8.60 & 8.60 & 8.80 & 8.70 & 8.70 & 8.60 \\
			2000 & 15.50 & 15.50 & 15.50 & 15.50 & 15.40 & 16.00 & 15.80 & 15.80 & 15.80 & 16.00 & 15.80 & 15.80 & 15.90 \\
			3000 & 25.90 & 25.80 & 25.50 & 25.60 & 25.60 & 27.10 & 27.00 & 26.70 & 26.70 & 27.10 & 27.10 & 26.90 & 26.90 \\
			3600 & 30.10 & 30.30 & 30.40 & 30.40 & 30.30 & 31.80 & 31.70 & 31.50 & 31.50 & 31.70 & 31.80 & 31.70 & 31.70 \\
			4400 & 36.70 & 36.50 & 36.20 & 36.30 & 36.30 & 37.50 & 37.40 & 36.90 & 37.00 & 37.40 & 37.40 & 36.80 & 37.00 \\
			5200 & 43.70 & 43.60 & 42.70 & 43.00 & 43.20 & 44.10 & 44.10 & 43.50 & 43.70 & 44.00 & 44.10 & 43.50 & 43.60 \\
			5600 & 46.60 & 45.80 & 46.30 & 46.50 & 46.20 & 47.00 & 47.10 & 46.60 & 46.80 & 46.80 & 47.00 & 46.60 & 46.80 \\
			6000 & 48.50 & 48.40 & 47.90 & 47.70 & 48.30 & 48.80 & 48.70 & 48.20 & 48.50 & 48.60 & 48.80 & 48.40 & 48.70 \\
			6400 & 49.40 & 49.30 & 49.10 & 49.30 & 49.10 & 49.80 & 49.50 & 49.10 & 49.40 & 49.60 & 49.70 & 49.10 & 49.40 \\
			\hline
			& $Y_{8j9}$ & $Y_{8j10}$ & $Y_{8j11}$ & $Y_{8j12}$ & $Y_{8j13}$ & $Y_{8j14}$ & $Y_{8j15}$ & $Y_{8j16}$ \\
			\hline
			1200 & 8.80 & 8.70 & 8.80 & 8.70 & 8.70 & 8.70 & 8.80 & 8.70 \\
			2000 & 16.00 & 15.90 & 16.00 & 16.00 & 15.90 & 15.90 & 16.00 & 15.80 \\
			3000 & 27.20 & 27.30 & 27.00 & 26.90 & 26.60 & 26.60 & 26.50 & 26.30 \\
			3600 & 31.90 & 31.90 & 31.90 & 31.80 & 31.70 & 31.60 & 31.50 & 31.50 \\
			4400 & 37.50 & 37.50 & 37.30 & 37.20 & 37.30 & 37.10 & 36.90 & 36.90 \\
			5200 & 44.30 & 44.40 & 43.80 & 43.90 & 44.00 & 43.80 & 43.50 & 43.50 \\
			5600 & 47.10 & 47.20 & 47.20 & 47.30 & 47.10 & 47.10 & 46.80 & 46.90 \\
			6000 & 49.00 & 49.10 & 49.10 & 49.20 & 49.00 & 49.00 & 48.60 & 48.60 \\
			6400 & 49.70 & 49.70 & 49.60 & 50.00 & 49.60 & 49.70 & 49.30 & 49.40 \\
		\end{longtable}
		\normalsize
	\end{center}
	\normalsize

\begin{table}[H]\centering
		\caption{Variances of the true engine power measurements ($\sxj$)
at the $jth$ engine rotation value.
}
		\begin{tabular}{rrrrrrrrr}
			\hline
			$ \sigma^2_{x_1}$ & $ \sigma^2_{x_2}$ & $ \sigma^2_{x_3}$ & $ \sigma^2_{x_4}$ & $ \sigma^2_{x_5}$ & $ \sigma^2_{x_6}$ & $ \sigma^2_{x_7}$ & $ \sigma^2_{x_8}$ & $ \sigma^2_{x_9}$ \\
			\hline
			0.0077 & 0.0256 & 0.0740 & 0.0999 & 0.1414 & 0.2007 & 0.2266 & 0.2500 & 0.2581 \\
			\hline
		\end{tabular}
	\end{table}
	\normalsize

\begin{table}[H]\centering
		\caption{Measurement error variances ($\sij$) for the $ith$
laboratory at the $jth$ engine rotation value. }
		\begin{tabular}{r|rrrrrrrrr}
			\hline
			& $j=1$ & $j=2$ & $j=3$ & $j=4$ & $j=5$ & $j=6$ & $j=7$ & $j=8$ & $j=9$ \\
			\hline
			$i=1$ & 0.0068 & 0.0215 & 0.0618 & 0.0848 & 0.1190 & 0.1690 & 0.1944 & 0.2141 & 0.2225 \\
			$i=2$ & 0.0054 & 0.0170 & 0.0491 & 0.0671 & 0.0949 & 0.1343 & 0.1535 & 0.1650 & 0.1711 \\
			$i=3$ & 0.0005 & 0.0018 & 0.0050 & 0.0069 & 0.0097 & 0.0136 & 0.0157 & 0.0169 & 0.0176 \\
			$i=4$ & 0.0081 & 0.0263 & 0.0750 & 0.1031 & 0.1446 & 0.2035 & 0.2333 & 0.2521 & 0.2615 \\
			$i=5$ & 0.0498 & 0.1587 & 0.4509 & 0.6270 & 0.8680 & 1.2158 & 1.3936 & 1.4954 & 1.5341 \\
			$i=6$ & 0.0101 & 0.0327 & 0.0935 & 0.1280 & 0.1806 & 0.2552 & 0.2888 & 0.3091 & 0.3186 \\
			$i=7$ & 0.0114 & 0.0372 & 0.1029 & 0.1435 & 0.2061 & 0.2919 & 0.3307 & 0.3591 & 0.3760 \\
			$i=8$ & 0.0249 & 0.0830 & 0.2371 & 0.3300 & 0.4543 & 0.6319 & 0.7243 & 0.7811 & 0.8060 \\
			\hline
		\end{tabular}
	\end{table}
	\normalsize

\newpage
\noindent
\medskip
\noindent{\Large{\bf Appendix B }}
\vspace{5mm}

\noindent
\medskip
{\large{\bf B.1}: Observed Information Matrix: $\frac{J^n({\bm\theta})}{n}$}.

\vspace{3mm}
After algebraic manipulations, the elements of the observed information matrix,
$\frac{1}{n} J^n(\mbox{\boldmath$\theta$})=-\frac{1}{n}\frac{\partial^2
L^n(\mbox{\boldmath$\theta$})}{\partial\mbox{\boldmath$\theta$}\partial\mbox{\boldmath$\theta$}^T}$
with
$\mbox{\boldmath$\theta$}=(\mu_{x_1},\cdots,\mu_{x_m},\alpha_2,\cdots,\alpha_p,\beta_2,\cdots,\beta_p)^T=(\theta_1,\cdots,\theta_{m+2(p-1)})^T$,
$J^n_{\theta_r,\theta_h}, r,h=1,\cdots,m+2(p-1)$
were obtained and are given by:

\vspace{3mm}
\noindent
$
J^n_{\mu_{x_j},\mu_{x_j}}=-\frac{(1-a_j^n)}{\sxj a_j^n},~
J^n_{\mu_{x_j},\mu_{x_q}}=0, ~J^n_{\mu_{x_j},\alpha_i}=
\frac{n_i \beta_i}{\sij a_j^n},~J^n_{\mu_{x_j},\beta_i}=
\frac{1}{\sij a_j^n}\lt(\frac{2 n_i \beta_i \sxj
	M_j^n}{a_j^n}-D_{ij}^n\rt),
$

\vspace{3mm}
\noindent
$
J^n_{\alpha_i,\alpha_i}=
\sum_{j=1}^m\frac{n_i}{\sij}\lt(1-\frac{n_i
	\beta_i^2 \sxj}{\sij a_j^n}\rt),~
J^n_{\alpha_i,\alpha_l}=
-\sum_{j=1}^m\frac{n_i n_l \beta_i \beta_l
	\sxj}{\sij \sigma^2_{lj} a_j^n},~
$

\vspace{3mm}
\noindent
$
J^n_{\alpha_i,\beta_i}=
\sum_{j=1}^m \frac{n_i \sxj}{\sij
	a_j^n}\lt[M_j^n-\frac{\beta_i}{\sij}\lt(\frac{2n_i\beta_i\sxj
	M_j^n}{a_j^n}-D_{ij}^n\rt)\rt],~
J^n_{\alpha_i,\beta_l}=
\sum_{j=1}^m\frac{n_i \beta_i \sxj}{\sij
	\sigma^2_{lj} a_j^n}\lt(D_{lj}^n-\frac{2 n_l \beta_l \sxj
	M_j^n}{a_j^n}\rt),~
$

\vspace{3mm}
\noindent
$
J^n_{\beta_i,\beta_i}=
\sum_{j=1}^m\frac{\sxj}{\sij
	a_j^n}\lt\{n_i-\frac{(D_{ij}^n)^2}{\sij}+\frac{n_i
	\sxj}{a_j^n}\lt[(M_j^n)^2\lt(1-\frac{4 n_i\sxj\beta_i^2}{\sij
	a_j^n}\rt)+\frac{4\beta_i M_j^n D_{ij}^n}{\sij}-\frac{2n_i
	\beta_i^2}{\sij}\rt]\rt\},~
$

\vspace{3mm}
\noindent
$
J^n_{\beta_i,\beta_l}=
-\sum_{j=1}^m \frac{\sxj}{\sij \sigma^2_{lj}
	a_j^n}\lt\{D_{ij}^n D_{lj}^n+\frac{2\sxj}{a_j^n}\lt[n_i n_l \beta_i
\beta_l\lt(1+\frac{2\sxj (M_j^n)^2}{a_j^n}\rt)-n_i\beta_i M_j^n D_{lj}^n-n_l \beta_l M_j^n
D_{ij}^n\rt]\rt\},~
$

\vspace{2mm}
\noindent
$j \neq q,$ $i \neq l$, $j,q=1,\cdots,m,~i,l=2,\cdots,p$ and
with $a_j^n$, $D_{ij}^n$ and $M_j^n$ as given in
Section 2.

\vspace{5mm}
\noindent
\medskip
{\large{\bf B.2}: Convergence of the elements of the observed information matrix, $\frac{J^n({\bm\theta})}{n}$, to the random matrix $ W({\bm\theta}). $ }

\vspace{3mm}
\medskip
We recall some notation introduced in Sections 2 and 3. Let $ \mbox{\boldmath$\psi$}^n \in \Theta$ be a vector of parameters such that
$ \mbox{\boldmath$\psi$}^n = \mbox{\boldmath$\theta$} + \frac{1}{\sqrt{n}}{\bf s}$ for some vector ${\bf s} \in \mathbb{R}^{m+2(p-1)}$ fixed and, let $ \mbox{\boldmath$\phi$}^n$ be a random vector satisfying

$$\mbox{\boldmath$\phi$}^n = (1-\delta^n) \mbox{\boldmath$\theta$} + \delta^n \mbox{\boldmath$\psi$}^n = (1-\delta^n) \mbox{\boldmath$\theta$} + \delta^n \left( \mbox{\boldmath$\theta$} + \frac{1}{\sqrt{n}}{\bf s} \right) = \mbox{\boldmath$\theta$} + \delta^n \frac{1}{\sqrt{n}}{\bf s},$$ where $0 < \delta^n < 1$ and $\delta^n$ is a random variable. Let $c=\max\{ \mid s_1 \mid , \cdots , \mid s_{m+2(p-1)}\mid \}$ be the norm of the vector ${\bf s}$. To simplify the notation, without loss of generality we assume that  $ \mbox{\boldmath$\psi$}^n = \mbox{\boldmath$\theta$} + \frac{c}{\sqrt{n}}$ and  $\mbox{\boldmath$\phi$}^n = \mbox{\boldmath$\theta$} + \delta^n \frac{c}{\sqrt{n}}$. We say that $n \rightarrow \infty$ as $n_i \rightarrow \infty$ and $\frac{n_i}{n} \rightarrow w_i$ where $w_i$ is a positive constant for every $i=1, \cdots , p$. We emphasize that the results are valid for every $\mbox{\boldmath$\theta$}^n$ such that $$ \sqrt{n} \mid \mbox{\boldmath$\theta$}^n - \mbox{\boldmath$\theta$} \mid \leq c, \quad n \geq 1.$$

For every ${\bm \theta} = (\mu_{x_1} , \cdots , \mu_{x_m} , \alpha_2, \cdots , \alpha_p , \beta_2 , \cdots , \beta_p) \in \mathbb{R}^{m+2(p-1)}$, the components of the observed information matrix depends on the following elements

\[
a_j^n ({\bm \theta} )= 1+ \sigma^2_{x_j} \left[ \frac{n_1}{\sigma^2_{1j}} + \sum_{i=2}^p \frac{n_i \beta_i^2}{\sigma^2_{ij}}  \right], ~ ~ M_j^n({\bm \theta} ) = \frac{\mu_{x_j}}{\sigma^2_{x_j}} + \sum_{k=1}^{n_1} \frac{Y_{1jk}}{\sigma^2_{1j}} +\sum_{i=2}^p \frac{\beta_i}{\sigma^2_{ij}} \left( \sum_{k=1}^{n_i} Y_{ijk} - n_i \alpha_i  \right),
\]
\[
D_{1j} ({\bm \theta} ) = \sum_{k=1}^{n_1} Y_{1jk} \quad \text{and} \quad D_{ij} ({\bm \theta} ) = \sum_{k=1}^{n_i} Y_{ijk} - n_i \alpha_i,\quad i=2, \cdots , p  \quad \text{and} \quad j=1, \cdots , m.
\]

\newpage
In this section, we will prove that

\[
\mathbb{P}_{\mbox{\boldmath$\psi$}^n , (x_1, \cdots , x_m)} \left[ \Big| \frac{1}{n} {\bf s}^T J^n({\bm \phi}^n) {\bf s} -  {\bf s}^T W(\mbox{\boldmath$\theta$}) {\bf s} \Big| \geq \epsilon \right] \rightarrow 0, \quad n \rightarrow \infty, \quad \epsilon > 0.
\] In order to prove the uniform convergence in probability, we will prove that each component of the observed information matrix converges uniformly in probability. We have that

\[\frac{{a_j^n (\mbox{\boldmath$\phi$}^n )}}{n} = \frac{1}{n}+ \sigma^2_{x_j} \left[\sum_{i=2}^{p}\frac{n_i}{n}\frac{(\beta_i + \delta^n \frac{c}{\sqrt{n}})^{2}}{\sigma_{ij}^{2}} + \frac{n_1}{n} \frac{1}{\sigma_{1j}^{2}} \right]  =
\]

 \[\frac{1}{n}+\sigma^2_{x_j} \left\{ \sum_{i=2}^{p} \frac{n_i}{n} \frac{1}{\sigma_{ij}^{2}} \left[ \beta_i^2 + 2 \beta_i  \frac{\delta^n c}{\sqrt{n}} + \left(  \frac{\delta^n c}{\sqrt{n}} \right)^2 \right] + \frac{n_1}{n} \frac{1}{\sigma_{1j}^{2}} \right\}=
\]

\[
\frac{1}{n}+ \sigma^2_{x_j}\sum_{i=2}^{p} \frac{n_i}{n} \frac{1}{\sigma_{ij}^{2}} \left[ 2 \beta_i  \frac{\delta^n c}{\sqrt{n}} + \left(  \frac{\delta^n c}{\sqrt{n}} \right)^2 \right] + \frac{n_1}{n}  \frac{\sigma^2_{x_j}}{\sigma_{1j}^{2}} + \sigma^2_{x_j}\sum_{i=2}^{p}\frac{n_i}{n}\frac{\beta_i^2}{\sigma_{1j}^{2}} .
\] As $\delta^n$ is a random variable such that $0 < \delta^n < 1$, we obtain that

\[
\Big| \frac{{a_j^n (\mbox{\boldmath$\phi$}^n )}}{n} - w_1 \frac{\sigma^2_{x_j}}{\sigma_{1j}^{2}} - \sigma^2_{x_j}\sum_{i=2}^{p} w_i \frac{\beta_i^2}{\sigma_{ij}^{2}}  \Big| \leq
\frac{1}{n} +  \sigma^2_{x_j}\sum_{i=2}^{p}\frac{n_i}{n} \frac{1}{\sigma_{ij}^{2}} \left[ \frac{c^2}{n}  +2 \frac{c \mid \beta_i \mid}{\sqrt{n}}  \right] +
\]

\[
\mid\frac{n_1}{n} - w_1 \mid \frac{\sigma^2_{x_j}}{\sigma_{1j}^{2}} + \sigma^2_{x_j}\sum_{i=2}^{p}\mid  \frac{n_i}{n} - w_i \mid \frac{\beta_i^2}{\sigma_{ij}^{2}}\] As a consequence, we obtain that

\[
\mathbb{P}_{\mbox{\boldmath$\psi$}^n, (x_1, \cdots , x_m)} \left[  \Big| \frac{{a_j^n (\mbox{\boldmath$\phi$}^n )}}{n} -\sigma^2_{x_j} \left( w_1 \frac{1}{\sigma_{1j}^{2}} + \sum_{i=2}^{p} w_i \frac{\beta_i^2}{\sigma_{ij}^{2}}\right) \Big| \geq \epsilon \right] ~~\overset{n\rightarrow\infty}\longrightarrow~~0, \quad \epsilon > 0,
\] and $j=1, \cdots , m$. In the sequel, we take the observed data ${\bf Y}^n$ with distribution $\mathbb{P}_{\mbox{\boldmath$\psi$}^n, (x_1, \cdots , x_m)}$, for every $n \geq 1$.  By definition, we have that for every $i=2, \cdots , p$,
\[
		\frac{\sum\limits_{k=1}^{n_i} Y_{ijk}}{n}
		=\frac{\sum\limits_{k=1}^{n_i}\left[ (\alpha_{i} + \frac{c}{\sqrt{n}}) + (\beta_{i} + \frac{c}{\sqrt{n}}) x_{j} + e_{ijk} \right]}{n_i} \frac{n_i}{n} = \frac{n_i}{n} \left[ \alpha_{i}+\beta_{i} x_{j} \right] + \frac{n_i}{n} \frac{\sum\limits_{k=1}^{n_i} e_{ijk} }{n_i} +  \frac{n_i}{n}  \frac{c}{\sqrt{n}} \left(1 + x_j \right) .
\] Then, we conclude that

\[
\left| \frac{\sum\limits_{k=1}^{n_i} Y_{ijk}}{n} - w_i \left( \alpha_{i}+\beta_{i} x_{j} \right) \right| \leq \mid \frac{n_i}{n} - w_i \mid \left| \alpha_{i}+\beta_{i} x_{j} \right|  + \frac{n_i}{n} \left|  \frac{\sum\limits_{k=1}^{n_i} e_{ijk} }{n_i} \right| +  \frac{n_i}{n}  \frac{c}{\sqrt{n}} \left|1 + x_j \right|, \quad i=2, \cdots , p.
\] By applying the same arguments, we obtain that

\[
\left| \frac{\sum\limits_{k=1}^{n_1} Y_{1jk}}{n} - w_1 x_{j} \right| \leq \mid \frac{n_1}{n} - w_1 \mid \left| x_{j} \right|  + \frac{n_1}{n} \left|  \frac{\sum\limits_{k=1}^{n_1} e_{1jk} }{n_1} \right| +  \frac{n_1}{n}  \frac{c}{\sqrt{n}} \left|1 + x_j \right|.
\]

As the distribution of the random error $\{e_{ij\ell} : \ell=1, 2, \cdots\}$ is independent of the parameter, the law of large number yields

\[
\mathbb{P}_{\mbox{\boldmath$\psi$}^n, (x_1, \cdots , x_m)} \left[ \left| \frac{\sum\limits_{k=1}^{n_i} e_{ijk} }{n_i} \right| \geq \epsilon \right] ~~\overset{n\rightarrow\infty}\longrightarrow~~ 0, \quad \epsilon > 0, \quad i=1, \cdots , p \quad \text{and} \quad j=1, \cdots m .
\] Hence, we conclude that

\begin{align*}
\mathbb{P}_{\mbox{\boldmath$\psi$}^n, (x_1, \cdots , x_m)} \left[ \left| \frac{\sum\limits_{k=1}^{n_i}Y_{ijk} }{n} -  w_i ~ \left( \alpha_{i}+ \beta_{i}x_{j} \right) \right| \geq \epsilon \right]		&  \overset{n\rightarrow\infty} \longrightarrow 0, \quad \epsilon > 0,
		\end{align*}  $i=2, \cdots , p ~ ~ \text{and} ~ ~ j=1, \cdots , m$, and

\begin{align*}
\mathbb{P}_{\mbox{\boldmath$\psi$}^n, (x_1, \cdots , x_m)} \left[ \left| \frac{\sum\limits_{k=1}^{n_1}Y_{1jk} }{n} -  w_1 x_{j} \right| \geq \epsilon \right]		&  \overset{n\rightarrow\infty} \longrightarrow 0, \quad \epsilon > 0,
		\end{align*}  for every  $j=1, \cdots , m$. In the sequel, for every $i=2, \cdots , p$, we have that

\[
\frac{{D_{ij}^n} (\mbox{\boldmath$\phi$}^n)}{n}=\frac{\sum\limits_{k=1}^{n_i}Y_{ijk} }{n}-\frac{n_i ~(\alpha_i + \frac{\delta^n c}{\sqrt{n}})}{n} \quad \text{and} \quad \frac{{D_{1j}^n} (\mbox{\boldmath$\phi$}^n)}{n}=\frac{\sum\limits_{k=1}^{n_1}Y_{ijk} }{n}.
\] Then, we obtain that

\[
\mathbb{P}_{\mbox{\boldmath$\psi$}^n, (x_1, \cdots , x_m)} \left[ \left| \frac{{D_{ij}^n} (\mbox{\boldmath$\phi$}^n)}{n} - w_i \beta_i x_j \right| \geq \epsilon  \right] \overset{n\rightarrow\infty} \longrightarrow 0 \quad \text{and} \quad \mathbb{P}_{\mbox{\boldmath$\psi$}^n, (x_1, \cdots , x_m)} \left[ \left| \frac{{D_{ij}^n} (\mbox{\boldmath$\phi$}^n)}{n} - w_1 x_j \right| \geq \epsilon  \right] \overset{n\rightarrow\infty} \longrightarrow 0.
\] By applying the same arguments, we have that

\[
\frac{M_j^n ( \mbox{\boldmath$\phi$}^n)}{n}  =  \frac{\mu_{x_j} + \frac{\delta^n c}{\sqrt{n}}}{n ~ \sxj} + \sum\limits_{i=2}^p \left(\beta_i + \frac{\delta^n c}{\sqrt{n}} \right) \frac{D_{ij}^n(\mbox{\boldmath$\phi$}^n)}{n \sij } +  \frac{D_{1j}^n(\mbox{\boldmath$\phi$}^n)}{n \sigma_{1j}^2 }, \quad j=1, \cdots , m.
\] Then, we conclude that

\[
\mathbb{P}_{\mbox{\boldmath$\psi$}^n, (x_1, \cdots , x_m)} \left[ \left| \frac{M_j^n ( \mbox{\boldmath$\phi$}^n)}{n}  - x_j \left(\sum\limits_{i=2}^p \frac{w_i ~ \beta_i^2 }{\sij} + \frac{w_1}{\sigma^2_{1j}} \right) \right| \geq \epsilon \right] \overset{n\rightarrow\infty} \longrightarrow 0 , \quad j=1, \cdots , m.
\] As a consequence of the continuous mapping theorem, we conclude that

\[
\mathbb{P}_{\mbox{\boldmath$\psi$}^n, (x_1, \cdots , x_m)} \left[ \left| \frac{{M_j^n (\mbox{\boldmath$\phi$}^n)}}{{a_j^n (\mbox{\boldmath$\phi$}^n )}} - \frac{x_j \left(\sum\limits_{i=2}^p \frac{w_i ~ \beta_i^2 }{\sij} + \frac{w_1}{\sigma^2_{1j}} \right) }{\sigma^2_{x_j} \left( w_1 \frac{1}{\sigma_{1j}^{2}} + \sum_{i=2}^{p} w_i \frac{\beta_i^2}{\sigma_{ij}^{2}}\right)}  \right| \geq \epsilon   \right] \overset{n\rightarrow\infty} \longrightarrow 0 , \quad j=1, \cdots , m.
\]

\medskip
In the sequel, we will show the uniform convergence for each component of observed information matrix. We have that

\medskip
\begin{enumerate}[1)]
	\item 
	{\large$ \frac{J_{(\mu_{x_j},\mu_{x_j})}(\mbox{\boldmath$\phi$}^n)}{n}=-\frac{(1-{a_j^n (\mbox{\boldmath$\phi$}^n)})}{n\sxj {a_j^n (\mbox{\boldmath$\phi$}^n)}} =-\frac{1}{\sxj }(\frac{1}{n}-\frac{{a_j^n (\mbox{\boldmath$\phi$}^n)}}{n}) \frac{1}{{a_j^n (\mbox{\boldmath$\phi$}^n)}}.$} Then, we obtain that \\
\[
\mathbb{P}_{\mbox{\boldmath$\psi$}^n, (x_1, \cdots , x_m)} \left[ \left| \frac{J_{(\mu_{x_j},\mu_{x_j})}(\mbox{\boldmath$\phi$}^n)}{n} \right|  \geq \epsilon \right] \overset{n\rightarrow\infty} \longrightarrow 0 , \quad \epsilon > 0 , \quad j=1, \cdots , m.
\] In this case, we have that $W_{(\mu_{x_j},\mu_{x_j})}(\mbox{\boldmath$\theta$}) = 0$, for every $j=1, \cdots , m$.\\ \\
\item 
{\large$ \frac{J_{(\mu_{x_j},\mu_{x_h})}(\mbox{\boldmath$\phi$}^n)}{n}=\frac{0}{n}~~\overset{n\rightarrow\infty}\longrightarrow~~0=  W_{(\mu_{x_j},\mu_{x_h})} (\mbox{\boldmath$\theta$})$,} for every $j \neq h=1, \cdots , m$.\\ \\
	\item 
{\large$ \frac{J_{(\mu_{x_j},  \alpha_i)  } (\mbox{\boldmath$\phi$}^n) }{n}=\frac{1}{n}\frac{n_i \left(\beta_i + \delta^n \frac{c}{\sqrt{n}}\right)}{\sij {a_j^n} (\mbox{\boldmath$\phi$}^n)} =\frac{n_i}{n} \frac{ \left(\beta_i + \delta^n \frac{c}{\sqrt{n}}\right)}{\sij }\frac{1}{{a_j^n (\mbox{\boldmath$\phi$}^n)}} ,$} for every $i=2, \cdots , p$ and $j=1, \cdots , m$. \\ \\
%

\noindent
As $a_j^n (\mbox{\boldmath$\phi$}^n) \rightarrow \infty$, we obtain that\\
\[
\mathbb{P}_{\mbox{\boldmath$\psi$}^n, (x_1, \cdots , x_m)} \left[ \left|\frac{J_{(\mu_{x_j},  \alpha_i)  } (\mbox{\boldmath$\phi$}^n) }{n} \right| \geq \epsilon \right]~~ \overset{n\rightarrow\infty}\longrightarrow ~~ 0, \quad \epsilon > 0, ~  i=2, \cdots , p, ~ j=1, \cdots , m.
\]

So, $W_{(\mu_{x_j},\alpha_i)}(\mbox{\boldmath$\theta$}) = 0$, for every  $i=2, \cdots , p$ and $j=1, \cdots , m$. \\ \\
\end{enumerate}

Using the arguments developed earlier, we obtain that
\medskip
\medskip


{\large $ W_{(\mu_{x_j},  \B_i)  }(\mbox{\boldmath$\theta$})=0; \quad W_{(\alpha_i\alpha_i) }(\mbox{\boldmath$\theta$})=
\sum\limits_{j=1}^{m} \frac{1}{\sdos_{ij}} -\B _i^2 \sum\limits_{j=1}^{m}\frac{1}{(\sdos_{ij})^2 \left( \sum\limits_{q=1}^{p} \frac{ \B _q^2 }{\sdos_{qj}}   \right) }; $}\\

{\large $  W_{(\alpha_i\alpha_l) }(\mbox{\boldmath$\theta$})=
-\B _i\B _l \sum\limits_{j=1}^{m}\frac{1}{\sdos_{ij}\sdos_{lj} \left( \sum\limits_{q=1}^{p} \frac{ \B _q^2 }{\sdos_{qj}}   \right) };\quad
W_{(\alpha_i\B_i) }=
\sum\limits_{j=1}^{m}\frac{x_j}{\sdos_{ij}}		-\B _i^2 \sum\limits_{j=1}^{m}\frac{x_j}{(\sdos_{ij})^2 \left( \sum\limits_{q=1}^{p} \frac{ \B _q^2 }{\sdos_{qj}}   \right) } ; $}\\

{\large $W_{(\alpha_i\B_l )}=
-\B _i\B _l \sum\limits_{j=1}^{m}\frac{x_j}{\sdos_{ij}\sdos_{lj} \left( \sum\limits_{q=1}^{p} \frac{ \B _q^2 }{\sdos_{qj}}   \right) };
\quad W_{(\B_i\B_i) }=
\sum\limits_{j=1}^{m} \frac{x_j^2}{\sdos_{ij}} \left\{1-\frac{\B _i^2}{\sdos_{ij}
	\left( \sum\limits_{q=1}^{p} \frac{ \B _q^2 }{\sdos_{qj}}   \right) } \right\}; $}\\

{\large $W_{(\B_i\B_l) }=
-\B _i\B _l \sum\limits_{j=1}^{m}\frac{x_j^2}{\sdos_{ij}\sdos_{lj} \left( \sum\limits_{q=1}^{p} \frac{ \B _q^2 }{\sdos_{qj}}   \right) }  ;~j=1,\cdots,m;~i,l=2,\cdots,p;~ i\neq l . $}\\

	\bibliographystyle{apalike}
	\bibliography{references24}
	\end{document}